\documentclass[english]{article}%
\usepackage{amsfonts}
\usepackage{amsmath}
\usepackage[T1]{fontenc}
\usepackage[latin9]{inputenc}
\usepackage{geometry}
\usepackage{amsthm}
\usepackage{amstext}
\usepackage{amssymb}
\usepackage{graphicx}
\usepackage{esint}
\usepackage[ruled,vlined]{algorithm2e}
\usepackage{babel}%
\usepackage{sidecap}
\usepackage{multirow}
\usepackage{array}
\makeatletter
\numberwithin{equation}{section}
\numberwithin{figure}{section}
\theoremstyle{plain}

\newtheorem*{lem*}{\protect\lemmaname}
\theoremstyle{plain}
\newtheorem*{prop*}{\protect\propositionname}
\makeatother
\providecommand{\lemmaname}{Lemma}
\providecommand{\propositionname}{Proposition}

\newtheorem{theorem}{Theorem}

\newtheorem{lemma}[theorem]{Lemma}

\newtheorem{proposition}[theorem]{Proposition}

\usepackage[affil-it]{authblk}

\begin{document}

\title{Exact Sampling of Stationary and Time-Reversed Queues}
\author{Jose Blanchet and Aya Wallwater }
\affil{Columbia University}
\date{}
\maketitle
\begin{abstract}
We provide the first algorithm that, under minimal assumptions, allows to
simulate the stationary waiting-time sequence of a single-server queue
backwards in time, jointly with the input processes of the queue
(inter-arrival and service times). The single-server queue is useful in
applications of DCFTP (Dominated Coupling From The Past), which is a well
known protocol for simulation without bias from steady-state distributions.
Our algorithm terminates in finite time assuming only finite mean of the
inter-arrival and service times. In order to simulate the single-server
queue in stationarity until the first idle period in finite expected
termination time we require the existence of finite variance. This
requirement is also necessary for such idle time (which is a natural
coalescence time in DCFTP applications) to have finite mean. Thus, in this
sense, our algorithm is applicable under minimal assumptions.
\end{abstract}

\markboth{J. Blanchet and A. Wallwater}{Exact Sampling of Stationary and
Time-Reversed Queues}

\section{Introduction}

It is a pleasure to contribute to this special issue in honor of Professor
Don Iglehart, whose scientific contributions have had an enormous impact in
the applied probability and stochastic simulation communities. Professor
Iglehart research contributions expand areas such as steady-state simulation
and queueing analysis. We are glad, in this paper, to contribute to both of
these areas from the standpoint of exact (also known as perfect) simulation
theory, which aims at sampling without any bias from the steady-state
distribution of stochastic systems.

The theory of exact simulation has attracted substantial attention,
particularly since the ground breaking paper \cite{PW96}. In their paper,
the authors introduced the most popular sampling protocol for exact
simulation to date; namely, Coupling From The Past (CFTP). CFTP is a
simulation technique which results in samples from the steady-state
distribution of a Markov chain under certain compactness assumptions. The
paper \cite{Ken98} describes a useful variation of CFTP, called Dominated
CFTP (DCFTP). Like CFTP, DCFTP aims to sample from the steady-state
distribution of a Markov chain, but this technique can also be applied to
cases in which the state-space is unbounded.

The idea in the DCFTP method is to simulate a dominating stationary process
backwards in time until the detection of a so-called coalescence time, in
which the target and dominating processes coincide. The sample path of the
target process can then be reconstructed forward in time from coalescence up
to time zero. The state of the target process at time zero is a sample from
the associated stationary distribution.

Our contribution in this paper is to provide, under nearly minimal
assumptions (finite-mean service and inter-arrival times), an exact
simulation algorithm for the stationary workload of a single-server queue
backwards in time. This is a fundamental queueing system which can be used
in many applications as a natural dominating process when applying DCFTP.
Usually, additional assumptions, beyond the ones we consider here, have been
imposed to enable the simulation of the stationary single-server queue
backwards in time. For example, in \cite{Si11a,Si11b} the author takes
advantage of a single-server queue with Poisson arrivals for exact
simulation of a multi-server system; see also the recent work of \cite{CK14},
which dramatically improves the running time in \cite{Si11b},  but also requires the Poisson arrivals assumption.
In the paper 
\cite{BS11}, under the existence of a finite moment generating function for
the service times, the single-server queue, simulated backwards in time, is
used to sample from a general class of perpetuities. The paper \cite{BD12},
which builds upon the ideas in \cite{BS11}, also uses the single-server
queue backwards in time to sample the state descriptor of the infinite
server queue in stationarity; in turn, the infinite-server queue is used
to simulate loss networks in stationarity. 
Other example in which the
single-server queue arises as a natural dominating process occurs in the
setting of so-called multi-dimensional stochastic-fluid networks, see \cite{BC11}. Our contribution here allows to extend the applicability these
instances, in which the single-server queue has been used as a dominated
process under stronger assumptions than the ones we impose here. The
extensions are direct in most cases, the multi-server queue with general renewal arrivals requires the application of an additional coupling idea and
it is reported in \cite{BDP15}.

The first idle period (backwards in time starting from stationarity) is a
natural coalescence time when applying DCFTP. Therefore, we are specially
interested in an algorithm that has finite expected termination time to
simulate such first idle period. Moreover, it is well known that
finite-variance service times are necessary if the first idle period
(starting from stationarity) has finite expected time (this follows from
Wald's identity, \cite{Durrett} p. 178, and from Theorem 2.1 in \cite{As03},
p. 270). While our algorithm terminates with probability one imposing only
the existence of finite mean of service times and inter-arrival times, when
we assume finite variances we obtain an algorithm that has finite expected
running time (see Theorem \ref{Thm_MAIN} in Section \ref%
{sec:Implementation-of-the algorithm}).

Let us now provide the mathematical description of the problem we want to
solve. Consider a random walk $S_{n}=X_{1}+\ldots+X_{n}$ for $n\geq1$, and $%
S_{0}=0$. We assume that $(X_{k}:\,k\geq1)$ is a sequence of independent and
identically distributed (IID) random variables with 
\begin{equation}
\begin{array}{cccc}
EX_{k}=0 & \text{ and } & E\left\vert X_{k}\right\vert ^{\beta}<\infty & 
\text{ for some }\beta>1.%
\end{array}
\label{eq:assumption on X_k}
\end{equation}
As we indicated earlier, of special interest is the case $E\left\vert
X_{k}\right\vert ^{\beta}<\infty$ for some $\beta>2$. Now, for $\mu>0$ and $%
n\geq0$ we define the negative-drift random walk and its associated running
(forward) maximum by 
\begin{equation}
S_{n}\left( \mu\right) =S_{n}-n\mu\ \ \text{and}\ \ M_{n}=\max_{m\geq
n}\{S_{m}\left( \mu\right) -S_{n}\left( \mu\right) \},
\label{eq:drifted RW and Max}
\end{equation}
respectively. Note that the maximum is taken over an infinite time-horizon,
so the process $(M_{n}:n\geq0)$ is not adapted to the random walk $%
(S_{n}\left( \mu\right) :n\geq0)$. Our aim in this paper is to design an
algorithm that samples jointly from the sequence $\left( S_{n}\left(
\mu\right) ,\,M_{n}:\,0\leq n\leq N\right) $ for any finite $N$ (potentially
a stopping time adapted to $(S_{n}\left( \mu\right) ,M_{n}:n\geq0)$). Of
particular interest is the first idle time, $N=\min\{n\geq0:M_{n}=0\}$,
which can often be used as a coalescence time.

Note that if we define $W_{m}=M_{-m}$ for $m\leq0$, then we can easily
verify the so-called Lindley's recursion (see \cite{As03}, p. 92) namely 
\begin{equation}
M_{-m}=\left( M_{-m+1}+X_{-m}-\mu\right) ^{+}=\left(
W_{m-1}+X_{-m}-\mu\right) ^{+}=W_{m}, 
\end{equation}
and therefore $\left( W_{m}:m\leq0\right) $ corresponds to a single-server
queue waiting time sequence backwards in time; the sequence is clearly
stationary since the $M_{n}$'s are all equal in distribution. Simulating $%
(S_{n}\left( \mu\right) ,M_{n}:n\geq0)$ jointly allows to couple the single-server queue backwards in time with the driving sequence (i.e. the $X_{n}$%
's). Such coupling is required in the applications of the DCFTP method. 

The algorithm that we propose here extends previous work in \cite{EG00},
which shows how to simulate $M_{0}$ assuming the existence of the so-called
Cramer root (i.e. $\theta>0$ such that $E\left(\exp\left( \theta
X_{1}\right)\right) =1$). The paper \cite{BS11} explains how to simulate $%
(S_{n}\left( \mu\right) ,M_{n}:n\geq0)$ assuming a finite moment generating
function in a neighborhood of the origin. Multidimensional extensions, also
under the assumption of a finite moment generating function around the
origin, are discussed in \cite{BC11}.

\medskip

Our strategy for simulating the sequence $(S_{n}\left( \mu \right)
,M_{n}:n\geq 0)$ relies on certain \textquotedblleft upward
events\textquotedblright\ and \textquotedblleft downward
events\textquotedblright\ that occur at random times. These
\textquotedblleft milestone events\textquotedblright\ will be discussed in
Section \ref{sec:Construction S_n M_n}. In Section \ref{sec:Construction S_n
M_n} we will also present the high-level description of our proposed
algorithm, which will be elaborated in subsequent sections. Section \ref%
{Sec_Sampling_M0} explains \ how to simulate $M_{0}$ under the assumption
that $E\left\vert X_{k}\right\vert ^{\beta }<\infty $ for $\beta >2$. In
Section \ref{sec:Implementation-of-the algorithm} we built on our
construction for the sampling of $M_{0}$ to simulate the sequence $%
(S_{k}\left( \mu \right) ,M_{k}:k\leq n)$. Section \ref{Sec_Add_Cons} will
explain how to extend our algorithm to the case $E\left\vert
X_{k}\right\vert ^{\beta }<\infty $ for $\beta >1$ and also discuss
additional considerations involved in evaluating certain normalizing
constants. Finally, in Section \ref{Sec: Numerical} we will present a
numerical example that tests the empirical performance of our
proposed algorithm.


\section{Construction of $\left( S_{\lowercase{n}}\left( \protect\mu\right)
,M_{\lowercase{n}}:\,\lowercase{n}\geq0\right) $ via \textquotedblleft milestone
events\textquotedblright}

\label{sec:Construction S_n M_n}

We will describe the construction of a pair of sequences of stopping times
(with respect to the filtration generated by $(S_{n}\left(\mu\right):n\geq0)$), denoted by $%
(D_{n}:n\geq0)$ and $(U_{n}:n\geq1)$, which track certain downward and
upward milestones in the evolution of $\left( S_{n}\left( \mu\right)
:\,n\geq0\right) $. We follow similar steps as described in \cite{BS11}.
These \textquotedblleft milestone events\textquotedblright\ will be used in
the design of our proposed algorithm. The elements of the two stopping times
sequences interlace with each other (when finite) and their precise description follows next.



\begin{SCfigure}
\label{fig:Construction of stopping times}%
   \centering
    \includegraphics[width=0.5\textwidth]%
    {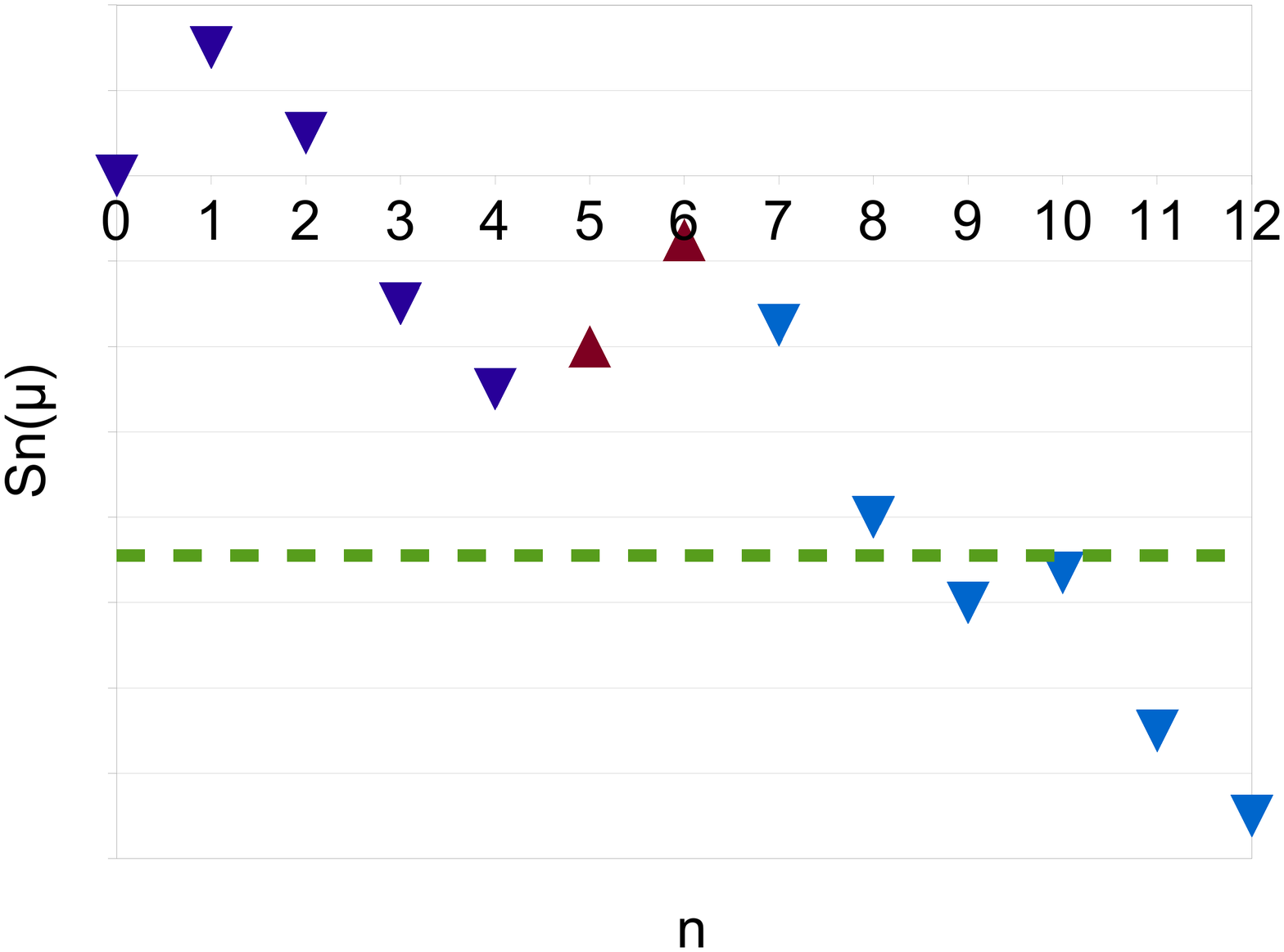}
    \vspace{-27pt}

    \caption{ \textbf{Figure \ref{fig:Construction of stopping times}}

   Figure
  \ref{fig:Construction of stopping times} illustrates a sample path $\left\lbrace S_n\left(\mu\right): \, 0\leq n\leq 12\right\rbrace$.
  If we set $m=1$ and $L=2$  then the corresponding stopping times  are $D_{1}=4$, $U_{1}=6$, $D_{2}=9$. 
  If in addition $U_{2}=\infty$, then $S_{n}\left(\mu\right)$ stays below the bold dashed line for all $n\geq D_2$. 
  Following Proposition \ref{pro:behavior of Delta_n and Gamma_n} we can now evaluate $M_n$ satisfying 
  $\lbrace M_n: n\leq 9, \, 
  S_n\left(\mu\right)\geq 
  S_9\left(\mu\right)+1 \rbrace
  $
  In this example, at time $t=D_2=9$ the values of $\left\lbrace M_n:0\leq n\leq 7\right\rbrace$ can be calculated and
  we can update $C_{UB}\leftarrow S_{D_{2}%
  }\left(  \mu\right)  +1$.
  Notice that $S_8\left(\mu\right)\leq S_9\left(\mu\right)+1$ and therefore, in order to determine $M_8$ we need to keep on tracking the path until the next time we spot $U_n =\infty$.
   }
\end{SCfigure}

We start by fixing any $m>0$, $L\geq 1$. Eventually we will choose $m$ as
small as possible subject to certain constraints described in Section \ref%
{Sec_Sampling_M0}, and then we can choose $L$ as small as possible to satisfy

\begin{equation}
P\left( m<M_{0}\leq (L+1)m\right) >0.  \label{C_m_0}
\end{equation}%
Typically, $L=1$ is feasible. This constraint on $L$ will be used in the
proof of Proposition \ref{pro:behavior of Delta_n and Gamma_n} and also in
the implementation of Step 2 in Procedure \ref{Algo Hueristics for S_n M_n}
below.

Now set $D_{0}=0$. We observe the evolution of the process $(S_{n}\left( \mu
\right) :\,n\geq 0)$ and detect the time $D_{1}$ (the first downward
milestone),
\[
D_{1}=\inf \left\{ n\geq D_{0}:\,S_{n}(\mu )<-Lm\right\} .
\]

Once $D_{1}$ is detected we check whether or not $\left\{ S_{n}\left(
\mu\right) :\,n\geq D_{1}\right\} $ ever goes above the height $%
S_{D_{1}}\left( \mu\right) +m$ (the first upward milestone); namely we
define
\[
U_{1}=\inf\left\{ n\geq D_{1}:\,S_{n}(\mu)>m+S_{D_{1}}\left( \mu\right)
\right\}
\]

For now let us assume that we can check if $U_{1}=\infty$ or $U_{1}<\infty$
(how exactly to do so will be explained in Section \ref{Sec_Sampling_M0}).
To continue simulating the rest of the path, namely $\left\{ S_{n}\left(
\mu\right) :\,n>D_{1}\right\} $, we potentially need to keep track of the
conditional upper bound implied by the fact that $U_{1}=\infty$. To this
end, we introduce the conditional upper bound variable $C_{UB}$ (initially $%
C_{UB}=\infty$). If at time $D_{1}$ we detect that $U_{1}=\infty$, then we
set $C_{UB}=S_{D_{1}}\left( \mu\right) +m$ and continue sampling the path of
the random walk conditional on never crossing the upper bound $%
S_{D_1}\left(\mu\right)+m$, that is, conditional on $\left\{ S_{n}\left(
\mu\right) <C_{UB}:\,n>D_{1}\right\} $. Otherwise, if $U_{1}<\infty$, we
simulate the path conditional on $U_{1}<\infty$, until we detect the time $%
U_{1}$. We continue on sequentially checking whenever a downward or an
upward milestone is crossed as follows: For $j\geq2$, define

\begin{equation}
\begin{array}{l}
D_{j}=\inf\left\{ n\geq U_{j-1}I\left( U_{j-1}<\infty\right) \vee
D_{j-1}:\,S_{n}\left( \mu\right) <S_{D_{j-1}}\left( \mu\right) -Lm\right\}
\\
U_{j}=\inf\left\{ n\geq D_{j}:\,S_{n}\left( \mu\right) -S_{D_{j}}\left(
\mu\right) >m\right\} ,%
\end{array}
\label{eq:construction of D_j and U_j}
\end{equation}
with the convention that if $U_{j-1}=\infty$, then $U_{j-1}I\left(
U_{j-1}<\infty\right) =0$. Therefore, we have that $U_{j-1}I\left(
U_{j-1}<\infty\right) >D_{j-1}$ if and only if $U_{j-1}<\infty$.

Let us define
\begin{equation}
\Delta=\inf\{D_{n}:U_{n}=\infty,n\geq1\}.  \label{eq def Delta}
\end{equation}
So, for example, if $U_{1}=\infty$ we have that $\Delta=D_{1}$ and the
drifted random walk will never reach level $S_{D_{1}}\left( \mu\right) +m$
again. This allows us to evaluate $M_{0}$ by computing
\begin{equation}
M_{0}=\max\left\{ S_{n}\left( \mu\right) :\,0\leq n\leq\Delta\right\} .
\label{Eq_EVAL_M0}
\end{equation}

Similarly, the event $U_{j}=\infty$, for some $j\geq1$, implies that the
level $S_{D_{j}}\left( \mu\right) +m$ is never crossed for all $n\geq D_{j}$%
, and we let $C_{UB}=S_{D_{j}}\left( \mu\right) +m$. The value of $C_{UB}$
keeps updating as the random walk evolves, at times where $U_{j}=\infty$.

The advantage of considering these stopping times is the following: once we
observed that some $U_{j}=\infty$, the values of $\left\{ M_{n}:\,n\leq
D_{j},S_{n}\left( \mu\right) \geq S_{D_{j}}\left( \mu\right)+m\right\} $ are
known without a need of further simulation. A detailed example is
illustrated in Figure \ref{fig:Construction of stopping times}.

Before we summarize the properties of the stopping times $D_n$'s and $U_n $%
's it will be useful to introduce the following. For any $a$ and $b>0$  let

\begin{equation}
\begin{array}{lc}
T_{b}  =\inf\left\{ n\geq0:\,S_{n}-\mu n>b\right\} , &  \\
T_{-b}  =\inf\left\{ n\geq0:\,S_{n}-\mu n<-b\right\}, & \\
P_a\left(\cdot\right)=P\left(\cdot \, \mid S_0=a\right).
\end{array}
\label{eq: Crossig times T_m T_(-m)}
\end{equation}

\medskip{}

\begin{proposition}
\label{pro:behavior of Delta_n and Gamma_n} Set $D_{0}=0$ and let $%
(D_{n}:\,n\geq1)$ and $(U_{n}:\,n\geq1)$ be as (\ref{eq:construction of D_j
and U_j}). We have that
\begin{equation}
\begin{array}{cccc}
P_{0}\left( \lim_{n\rightarrow\infty}D_{n}=\infty\right) =1 & \text{ and } &
P_{0}\left( D_{n}<\infty\right) =1,\, & \forall n\geq1.%
\end{array}
\label{eq:behavior of D_n}
\end{equation}
Furthermore,
\begin{equation}
P_{0}\left( U_{n}=\infty,\,\mathnormal{i.o.}\,\right) =1.
\label{eq: behavior U_n io as}
\end{equation}
\end{proposition}

\medskip{}

\begin{proof}
The statement in (\ref{eq:behavior of D_n}) follows easily from the Law of
Large Numbers since $ES_{1}\left( \mu \right) =-\mu <0$. 
Now we will verify that $%
P_{0}\left( U_{n}=\infty ,\,\mathnormal{i.o.}\,\right) =1$. Recall that $%
U_{1}$ was defined by $U_{1}=\inf \left\{ n\geq D_{1}:\,S_{n}(\mu
)-S_{D_{1}}\left( \mu \right) >m\right\} $. Therefore, since $ES_{1}(\mu )<0$%
, for all $m\geq 0$ we have (see \cite{As03} p. 224),
\[
\begin{array}{c}
P_{0}\left( U_{1}=\infty |S_{1},...,S_{D_{1}}\right) =P_{0}\left(
T_{m}=\infty \right) =P\left( M_{0}\leq m\right) \geq P\left( M_{0}=0\right)
>0.%
\end{array}%
\]%
Our next goal is to show that for $j\geq 2$ we can find $\delta >0$ such
that
\[
P_{0}\left( U_{j}=\infty |S_{1},...,S_{D_{j}},U_{1},...,U_{j-1}\right) \geq
\delta >0.
\]%
Suppose first that $U_{l}<\infty $ for each $l=1,2,...,j-1$. Then, by the
strong Markov property we have that
\[
P_{0}\left( U_{j}=\infty |S_{1},...,S_{D_{j}},U_{1},...,U_{j-1}\right)
=P_{0}\left( T_{m}=\infty \right) \geq P\left( M_{0}=0\right) >0.
\]%
Now suppose that $U_{l}=\infty $ for some $l\leq j-1$ and let $l^{\ast
}=\max \left\{ l\leq j-1:\,U_{l}=\infty \right\} $. Define $r=S_{D_{l^{\ast
}}}\left( \mu \right) +m-S_{D_{j}}\left( \mu \right) \geq (L+1)m$. Note that
\begin{equation}
P_{0}\left( U_{j}=\infty |S_{1},...,S_{D_{j}},U_{1},...,U_{j-1}\right)
=P_{0}\left( T_{m}=\infty |T_{r}=\infty \right) .
\label{eq: Pr T_m finite given T_M infinite}
\end{equation}%
Keep in mind that the right hand side of (\ref{eq: Pr T_m finite given T_M
infinite}) regards $r$ as a deterministic constant and note that
\begin{equation}
P_{0}\left( T_{m}=\infty |T_{r}=\infty \right) =P_{0}\left( M_{0}\leq
m|M_{0}\leq r\right) \geq \frac{P_{0}\left( M_{0}=0\right) }{P\left(
M_{0}\leq r\right) }\geq P_{0}\left( M_{0}=0\right) >0
\label{Good_Lower_Bound_p}
\end{equation}%
Hence, we conclude that%
\[
\begin{array}{c}
P_{0}\left( U_{j}=\infty |S_{1},...,S_{D_{j}},U_{1},...,U_{j-1}\right) \geq
P\left( M_{0}=0\right) :=\delta >0%
\end{array}%
.
\]%
It then follows by the Borel-Cantelli lemma that $P_{0}\left( U_{n}=\infty
,\,\text{\textnormal{i.o.}}\,\right) =1$.
\end{proof}

\medskip

In the setting of Proposition \ref{pro:behavior of Delta_n and Gamma_n}, for
each $k\geq0$ we can define $N_{0}\left( k\right) =\inf\left\{
n\geq1:\,D_{n}\geq k\right\} $ and $\mathcal{T}\left( k\right) =\inf\left\{
j\geq N_{0}\left( k\right) +1:\,U_{j}=\infty\right\} $, both finite random
variables such that%
\begin{equation}
M_{k}=-S_{k}\left(\mu\right)+ \max\lbrace S_{n}\left(
\mu\right):\,k\leq n\leq D_{\mathcal{T}\left( k\right)} \rbrace
\label{eq:connection M_n and stopping times}
\end{equation}


In words, $D_{\mathcal{T}\left( k\right) }$ is the time, not earlier than $k$%
, at which we detect a second unsuccessful attempt at building an upward\
patch directly. The fact that the relation in (\ref{eq:connection M_n and
stopping times}) holds, follows easily by construction of the stopping times
in (\ref{eq:construction of D_j and U_j}). Note that it is important,
however, to define $\mathcal{T}\left( k\right) \geq N_{0}\left( k\right) +1$
so that $D_{N_{0}\left( k\right) +1}$ is computed first. That way we can
make sure that the maximum of the sequence $\left( S_{n}\left( \mu\right)
:\,n\geq k\right) $ is achieved between $k$ and $D_{\mathcal{T}\left(
k\right) }$ (see Figure \ref{fig:Construction of stopping times}).

\medskip{}

Proposition \ref{pro:behavior of Delta_n and Gamma_n} ensures that it
suffices to sequentially simulate $(D_{n}:n\geq0)$ and $(U_{n}:n\geq1)$
jointly with the underlying random walk in order to sample from the sequence
$\left( S_{n}\left( \mu\right) ,\,M_{n}:\,n\geq0\right) $. This observation
gives rise to our suggested scheme. The procedure sequentially constructs
the random walk in the intervals $\left[ D_{n-1},D_{n}\right) $ for $n\geq1$%
. Here is the\textbf{\ }high-level procedure to construct $\left(
S_{n}\left( \mu\right) ,\,M_{n}:n\geq0\right) $:

\begin{figure}[tbp]
        \begin{minipage}[b]{0.43\textwidth}
                \includegraphics[scale=0.217]{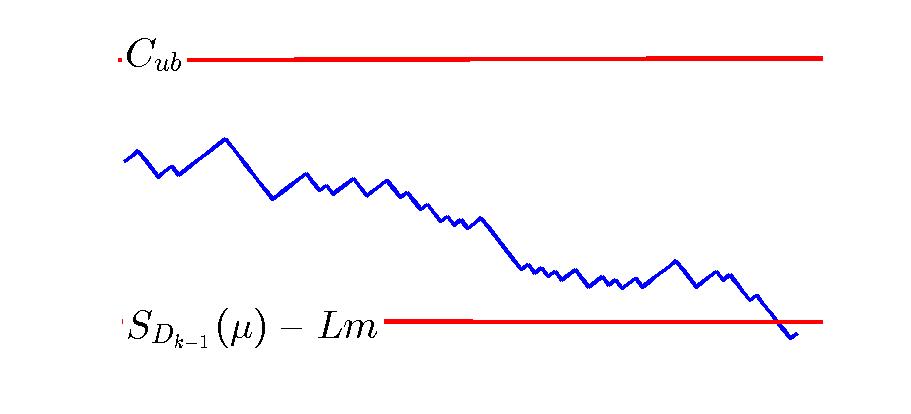}
                
        \begin{center} a. Step 1.
        \end{center}       
        \end{minipage}
        \begin{minipage}[b]{0.43\textwidth}
        \includegraphics[scale=0.217]{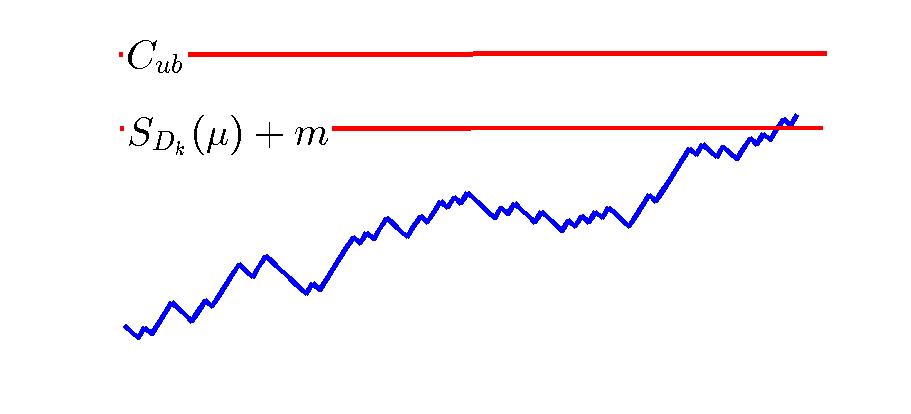}
        
        \begin{center}b. Step 2.
        \end{center}
        \end{minipage}
        \caption{High-level description of the algorithm}
\label{fig:Heuristics of algorithm}%
\end{figure}



\begin{procedure}[H]
\caption{1 () Milestone Construction of $\left(S_{n}\left(\mu \right),\, M_{n}: \,n\geq 0 \right)$ (see Figure \ref{fig:Heuristics of algorithm})}
\label{Algo Hueristics for S_n M_n}
At $k$th iteration, $k\geq1$: \\
{\bf {Step 1:  ``downward patch". } } Conditional on the path
not crossing $C_{UB}$ we simulate the path until we detect $D_{k}\,-$
the first time the path crosses the level $S_{D_{k-1}}(\mu)-Lm$ \label{Hueristics Step_1}
(see Figure \ref{fig:Heuristics of algorithm}a).\\
\smallskip{}
{\bf {Step 2: ``upward patch". }}  Check whether or not the
level $S_{D_{k}}(\mu)+m$ is ever crossed. That is, whether $U_{k}<\infty$
or not. If the answer is ``Yes'', then conditional on the path crossing
the level $S_{D_{k}}(\mu)+m$ but not crossing the level $C_{UB}$ we simulate
the path until we detect $U_{k}$, the first time the level $S_{D_{k}}(\mu)+m$
is crossed\label{Hueristics Step_2}
(see Figure \ref{fig:Heuristics of algorithm}b).
Otherwise $\left(U_{j}=\infty \right)$, and we can update $C_{UB}$: $C_{UB}\leftarrow S_{D_{j}}(\mu)+m$
\end{procedure}

\bigskip{}

The implementation of the steps in Procedure \ref{Algo Hueristics for S_n
M_n} will be discussed in detail in the next sections, culminating with the
precise description given in Algorithm \ref{algo last} at the end of
Subsection \ref{Sub_Algo_last}. The following result summarizes the main
contribution of this paper. The development in the next sections provides
the proof of this result, which will ultimately be given after the
description of Algorithm \ref{algo last}. 

Throughout the rest of the paper
a function evaluation is considered to be any of the following operations:
evaluation of a sum, a product, the exponential of a number, the underlying
increment distribution at a given point, the simulation of a uniform number, and the simulation of a single increment conditioned on lying on a given interval.

\begin{theorem}
\label{Thm_MAIN}Suppose that $E\left\vert X_{k}\right\vert ^{\beta}<\infty$
for some $\beta>1$. If $m>0$ is suitably chosen (see Subsection \ref%
{Sub_Select_m}) then for each $n\geq0$ deterministic it is possible to
simulate exactly the sequence $(D_{j}:0\leq j\leq n)$ and $(U_{j}:0\leq
j\leq n)$ jointly with $(S_{j}\left( \mu\right) :j\leq n)$ and therefore
(given our previous discussion on the evaluation of $M_{k}$), the sequence $%
(S_{k}\left( \mu\right) ,M_{k}:0\leq k\leq n)$.

Moreover, if $\beta>2$, the
expected number of function evaluations required to simulate $(S_{k}\left(
\mu\right) ,M_{k}:0\leq k\leq n)$ is finite. In particular, since $EN<\infty$
for $N=\inf\{k\geq0:M_{k}=0\}$, the expected running time to simulate $%
(S_{k}\left( \mu\right) ,M_{k}:0\leq k\leq N)$ is also finite.
\end{theorem}


\section{Sampling $M_0$ jointly with $\left( S_{1}\left( \protect\mu\right)
,...,S_{\Delta}\left( \protect\mu\right) \right) $}

\label{Sec_Sampling_M0}

The goal of this section is to sample exactly from the steady-state
distribution of the single-server queue, namely $M_{0}$. To this end we need
to simulate the sample path up to the first $U_{j}$ such that $U_{j}=\infty$ (recall that $\Delta$ was defined to be the corresponding $D_j$). 
This sample path will be used in the construction of further steps in
Procedure \ref{Algo Hueristics for S_n M_n}.

Throughout this section, in order to simplify the exposition, we will assume
that $E\left\vert X_{k}\right\vert ^{2+\varepsilon}<\infty$ (i.e. $%
\beta=2+\varepsilon$). This will allow us to conclude that our algorithm has
finite expected termination time. We will discuss the case $E\left\vert
X_{k}\right\vert ^{1+\varepsilon}<\infty$ only (for $\varepsilon\in(0,1)$)
in Section \ref{Sec_Add_Cons} for completeness, but in such case the
algorithm may take infinite expected time to terminate.

Let us recall the definition of the crossing stopping times $T_{b}$ and $%
T_{-b}$, for $b>0$, introduced in (\ref{eq: Crossig times T_m T_(-m)}).
Since we concentrate on $M_{0}$, we have that $C_{UB}=\infty$. We first need
to explain a procedure to generate a Bernoulli random variable with success
parameter $P_{0}\left( T_{m}<\infty\right) $, for suitably chosen $m>0$.
Also, this procedure, as we shall see will allow us to simulate $\left(
S_{1}\left( \mu\right) ,...,S_{T_{m}}\left( \mu\right) \right) $ given that $%
T_{m}<\infty$.

\subsection{Sampling $Ber\left( P_{0}\left( T_{m}<\infty\right) \right) $
and $\left( S_{1}\left( \protect\mu\right) ,...,S_{T_{m}}\left( \protect\mu%
\right) \right) $ given $T_{m}<\infty$}

\label{Sampling_Up_Path}

Let us denote by $J$  a Bernoulli random variable with success
parameter $P_{0}\left( T_{m}<\infty\right) $. The constant $m>0$ will be
selected below in Subsection \ref{Sub_Select_m}. There are several ways of
sampling $J$, we use a strategy similar to that considered in \cite{MJB13},
in connection to a different sampling problem.

In order to sample $J$ we first introduce a partition on the natural numbers
(i.e. the positive time line on the lattice) as follows.
Let
\begin{equation}
\begin{array}{cc}
n_{k}=2^{k-1}, & k=1,2,\ldots%
\end{array}
.  \label{eq:Step_2  natural numbers partition}
\end{equation}
This sequence define a partition of the natural numbers via the sets $%
[n_{k-1},n_{k}-1]$ for $k=2,3,...$. Now, for $k=2,3,\ldots$ we consider the
sets

\begin{equation}
\begin{array}{l}
A_{k}=\bigcup\limits_{j=n_{k-1}}^{n_{k}-1}\left\{ X_{j}>\left( \mu
j+m\right) ^{1-\delta }\right\} \\
B_{k}=\bigcap\limits_{j=1}^{n_{k}-1}\left\{ X_{j}\leq \left( \mu
n_{k-1}+m\right) ^{1-\delta }\right\} \vspace{0.15cm} \\
A_{k}^{c}\cap B_{k}^{c} \\
\end{array}
\label{eq:Step2 sets Ak Bk Ck}
\end{equation}%
for some $\delta \in (0,1/2]$, also to be selected.

First, the algorithm samples the random variable $K\geq 2$, which has
probability mass function $g(\cdot )$ that will be specified later. The
random variable $K$ relates to the partition on the natural numbers that was
induced by (\ref{eq:Step_2 natural numbers partition}) and $K=k$ will
eventually imply that $T_{m}\in \lbrack n_{k-1},n_{k}-1]$. Given $K=k$, the
algorithm then proposes a walk $\left( S_{1}\left( \mu \right) ,\ldots
,S_{n_{k}-1}\left( \mu \right) \right) $ via conditioning on one of three
possible events described in terms of $A_{k},\,B_{k}\cap A_{k}^{c}$ and $%
A_{k}^{c}\cap B_{k}^{c}$ with equal probability (i.e. $1/3$ each).
Conditioning on $A_{k}$ and $A_{k}^{c}\cap B_{k}^{c}$ will be handled using
a mixtures based on individual large-jum-events of the form $\{X_{j}>\left(
\mu j+m\right) ^{1-\delta }\}$. Conditioning on $B_{k}$ will be handled
using an exponential tilting of the distribution of $X_{j}$ given that $%
\{X_{j}<\left( \mu j+m\right) ^{1-\delta }\}$. The tilting parameter will be
selected via
\begin{equation}
\theta _{k}=\gamma /\left( n_{k-1}\mu +m\right) ,  \label{eq:sel_tilting_1}
\end{equation}%
for some $\gamma >0$.

In order to describe all of these conditional sampling procedures we need to
provide some definitions and state auxiliary lemmas which will be proved in
the appendix.

We will start by specifying the probability mass function $\{g\left(
k\right) ,\,k\geq 2\mathbb{\}}$. Consider $Y$, a Pareto distributed random
variable with some regularly varying index $\alpha >0$, namely,
\[
P\left( Y>y\right) =\frac{1}{\left( 1+y\right) ^{\alpha }},
\]%
for $y\geq 0$. Conditions on $\alpha >0$ will be imposed below. Let
\[
\bar{G}\left( t\right) =\int\limits_{t}^{\infty }P\left( Y>s\right) ds
\]%
Then we set for $k=2,3,\ldots $
\begin{equation}
g\left( k\right) =P\left( K=k\right) =\frac{\bar{G}\left( m+\mu
n_{k-1}\right) -\bar{G}\left( m+\mu n_{k}\right) }{\bar{G}\left( m+\mu
n_{1}\right) }.  \label{def of pmf g}
\end{equation}%
Let us impose conditions on $\delta ,\alpha ,m$ and $\gamma $ that will be
assumed for the implementation of the algorithm.

\subsubsection{Assumptions imposed on the parameters $\protect\delta $, $%
\protect\alpha $, $m$ and $\protect\gamma $\label{Sub_Select_m}}

In addition to $\delta \in (0,1/2]$, and (\ref{C_m_0}), assume that $m\geq 1$
is selected large enough so that
\begin{equation}
\frac{E\left( X^{2}\right) }{m^{2\left( 1-\delta \right) }}\leq \frac{1}{2},
\label{C_m_3}
\end{equation}%
and that the following inequalities hold:%
\begin{eqnarray}
\sup_{z\in \mu \cdot \{2^{k}:k\geq 0\}}\frac{6\left( 1+2z+m\right) ^{\alpha
}P\left( X>\left( z+m\right) ^{1-\delta }\right) }{\left( \alpha -1\right)
\left( m+1\right) ^{\alpha -1}\mu } &\leq &1,  \label{C_I_1} \\
\sup_{z\in \mu \cdot \{2^{k}:k\geq 0\}}\frac{\exp \left( -\gamma \left(
m+z\right) ^{\delta }+\frac{\gamma ^{2}e^{\gamma }E\left( X^{2}\right) z}{%
\left( m+z\right) ^{2(1-\delta )}\mu }+4\frac{z}{\mu }P\left( X>\left(
z+m\right) ^{1-\delta }\right) \right) }{3^{-1}\left( \alpha -1\right)
\left( m+1\right) ^{\alpha -1}\left(1+2z+m\right) ^{-\alpha }z} &\leq &1.
\label{C_I_2}
\end{eqnarray}%
Inequalities (\ref{C_I_1}) and (\ref{C_I_2}) are used during the proofs of
Lemmas \ref{lem 3PA_k over g_k leq 1} and \ref{lem PB_k over g_k leq 1} ,
respectively. Inequality (\ref{C_m_3}) appears in a simple technical step
leading to (\ref{C_I_2}).

In Appendix  \ref{Subsec_Discussion_Param} we will discuss how equations (\ref{C_m_3})-(\ref{C_I_2}) can always be satisfied under our assumptions on the increments $X_{k}$.

%

\subsubsection{Some technical lemmas underlying the description of our
algorithm}

Using the previous assumptions we now are ready to discuss a series of
technical lemmas that are the basis for our algorithm.

\begin{lemma}
\label{lem 3PA_k over g_k leq 1} Under (\ref{C_m_2})
 (see Appendix \ref{Subsec_Discussion_Param}) we have that

\begin{equation}
\frac{3P\left( A_{k}\right) }{g\left( k\right) }\leq1,\;\forall k\geq2.
\label{eq: ratio bound A_k}
\end{equation}
\end{lemma}

\begin{proof}
See Appendix \ref{appndx pf lem bound A_k}
\end{proof}

\bigskip

On the event $B_{k}$ we sample the path $\left( S_{1}\left( \mu \right)
,...,S_{n_{k}-1}\left( \mu \right) \right) $ using an exponential tilting.
Specifically, we sample the increments, $\left( X_{j}:1\leq j\leq
n_{k}-1\right) $, conditional on the event $B_{k}$ and tilted with parameter
$\theta _{k}$ up to time $\min \left\{ T_{m},n_{k}-1\right\} $, where
\[
\theta _{k}=\frac{\gamma }{C_{k}^{1-\delta }},\text{ \ and \ }C_{k}:=\left(
n_{k-1}\mu +m\right) .
\]%
Recall that $\gamma >0$ has been implicitly constrained due to (\ref{C_I_2}%
). The corresponding log-mgf is given by
\[
\psi _{k}\left( \theta _{k}\right) :=\log \left( \frac{E\left[ \exp \left\{
\theta _{k}X\right\} I\left( X\leq C_{k}^{1-\delta }\right) \right] }{%
P\left( X\leq C_{k}^{1-\delta }\right) }\right) .
\]%
The likelihood ratio between $P\left( X_{j}\in \cdot |X_{j}\leq
C_{k}^{1-\delta }\right) $ and the tilted distribution (to be used in an IID
way for $1\leq j\leq n_{k}-1$) denoted via $P_{k,1}\left( \cdot \right) $ is
given by
\begin{equation}
\frac{dP_{k,1}}{dP}\left( X\right) =\frac{I\left( X\leq C_{k}^{1-\delta
}\right) \exp \left( \theta _{k}X-\psi _{k}\left( \theta _{k}\right) \right)
}{P\left( X\leq C_{k}^{1-\delta }\right) }.  \label{eq def of P_k1}
\end{equation}

Now we summarize some bounds for this likelihood ratio.

\begin{lemma}
\label{lem PB_k over g_k leq 1} Under conditions (\ref{C_m_3})-(\ref{C_I_2}) we have that

\begin{equation}
\frac{3\exp(-\theta_{k}S_{T_{m}}+T_{m}\psi_{k}\left( \theta_{k}\right) )}{%
g\left( k\right) } \leq1,\;\forall k\geq2.
\label{eq: ratio bound B_k}
\end{equation}
\end{lemma}

\begin{proof}
See Appendix \ref{Appndx pf lemma bound exp tilting}
\end{proof}

\bigskip

As the final piece we will note the following.

\begin{lemma}
\label{lem 3PB_k^c over g_k leq 1} Then, under (\ref{Cond_Alpha}), and (\ref%
{C_m_2}) we have that%
\begin{equation}
\frac{3P\left( B_{k}^{c}\right) }{g\left( k\right) }\leq1,\;\forall k\geq2.
\label{eq: ratio bound B_k^c}
\end{equation}
\end{lemma}

\begin{proof}
See Appendix \ref{appndx pf lem bound B_k ^c}
\end{proof}

\bigskip

\subsubsection{Algorithm for sampling $Ber\left( P_{0}\left(
T_{m}<\infty\right) \right) $ jointly with $\left( S_{1}\left( \protect\mu%
\right) ,...,S_{T_{m}}\left( \protect\mu\right) \right) $ given $T_{m}<\infty$}

Now we are ready to fully discuss our algorithm to sample $J$ and $%
\omega=\left( S_{1},...S_{T_{m}}\right) $ given $T_{m}<\infty$. In addition
to the random variable $K$ following the probability mass function $g\left(
\cdot\right) $, let us introduce a random variable $Z$ uniformly distributed
on $\{0,1,2\}$ and independent of $K$. Finally, we also introduce $V\sim
U\left( 0,1\right) $ independent of everything else.

If $Z=0$, then we sample the path $(S_{1},...,S_{n_{k}-1})$ conditional on $%
A_{k}$ (denote $P_{k,0}\left( \cdot\right) =P\left( \cdot|\,A_{k}\right) $).
This will be explained in Subsection \ref{subsubpk0}, the sample takes $%
O\left( n_{k}\right) $ function evaluations to be produced. Then we let
\[
J=I(V\leq3P\left( A_{k}\right) I(T_{m}\in\left[ n_{k-1},n_{k}-1\right]
)/g\left( k\right) ).
\]
Owing to Lemma \ref{lem 3PA_k over g_k leq 1}, we have that

\begin{equation}
\frac{3P\left( A_{k}\right) }{g\left( k\right) }\leq1,\;\forall k\geq2.
\end{equation}

If $Z=1$, we sample $(S_{1}\left( \mu \right) ,...,S_{n_{k}-1}\left( \mu
\right) )$ by applying each increment $X_{j}$ conditional on $\{X_{j}\leq
\left( \mu n_{k-1}+m\right) ^{1-\delta }\}$ for $j\in \{1,...,n_{k}-1\}$ in
an IID way each following the exponential tilting (\ref{eq def of P_k1}).
This sampling distribution is denoted via $P_{k,1}\left( \cdot \right) $.
The simulation of each increment is done using Acceptance/Rejection, as we
shall explain, and the overall sampling $\{X_{j}:j\leq n_{k}-1\}$ takes $%
O\left( n_{k}\right) $ function evaluations, see Subsections \ref%
{subsubpk1_B}. Additional discussion on the evaluation $\psi _{k}\left(
\theta _{k}\right) $ in $O\left( n_{k}\right) $ function evaluations is
given in Subsection \ref{SubSam_P}. We then set%
\[
J=I(V\leq 3\cdot \exp \left\{ -\theta _{k}S_{T_{m}}+T_{m}\psi _{k}\left(
\theta _{k}\right) \right\} I(T_{m}\in \left[ n_{k-1},n_{k}-1\right]
,A_{k}^{c},B_{k})/g\left( k\right) ).
\]%
Observe that Lemma \ref{lem PB_k over g_k leq 1} guarantees the inequality

\begin{equation}
\frac{3\exp \left\{ -\theta _{k}S_{T_{m}}+T_{m}\psi _{k}\left( \theta
_{k}\right) \right\} }{g\left( k\right) }\leq 1,\;\forall k\geq 2.
\label{BND_PK2}
\end{equation}

Finally, if $Z=2$, we sample the path $(S_{1}\left( \mu \right)
,...,S_{n_{k}-1}\left( \mu \right) )$ conditional on the event $B_{k}^{c}$
(denote $P_{k,2}\left( \cdot \right) =P\left( \cdot |\,B_{k}^{c}\right) $).
This is done in a completely analogous manner as in Subsection \ref%
{subsubpk0}, thus taking $O\left( n_{k}\right) $ function evaluations. We
then let
\[
J=I(V\leq 3P\left( B_{k}^{c}\right) I(T_{m}\in \left[ n_{k-1},n_{k}-1\right]
,A_{k}^{c},B_{k}^{c})/g\left( k\right) ).
\]%
Here the inequality

\begin{equation}
\frac{3P\left( B_{k}^{c}\right) }{g\left( k\right) }\leq1,\;\forall k\geq2 ,
\end{equation}
is obtained thanks to Lemma \ref{lem 3PB_k^c over g_k leq 1}. \medskip

Upon termination we will output the pair $\left( J,\omega\right) $. If $J=1$%
, then we set $\omega=(S_{1}\left( \mu\right) ,...,S_{T_{m}}\left(
\mu\right) )$. Otherwise ($J=0$), we set $\omega=\left[ \,\right] $, the
empty vector. The precise description of the algorithm is given next.

\begin{algorithm}
\caption{Sampling $Ber\left( P_{0}\left( T_{m}<\infty \right) \right) $
and $\left( S_{1}\left(\mu \right),...,S_{T_{m}}\left( \mu \right)\right) $ given $T_{m}<\infty $
}\label{algo Sampling M_0}
\KwIn{ $g\left( \cdot \right)$ as in (\ref{def of pmf g}), with
$\alpha, \delta, m, \gamma$ satisfying the conditions in Section \ref{Sub_Select_m}
and $L$ as in (\ref{C_m_0}).}
\KwOut{ $J\sim Ber\left(P_0 \left(T_m <\infty \right)\right)$ and $\omega$. If $J=1$, then $\omega =\left(S_1\left(\mu \right),\ldots, S_{T_m} \left(\mu \right) \right)$. \newline Otherwise ($J=0$), $\omega =\left[ \, \right]$ \tcp{If $J=0$, then $\omega$ equals to the empty vector}
} Sample a time $K$ with probability mass function $g\left(k\right)=P%
\left(K=k\right)$\label{Sample time K}\newline
Sample $Z\sim Unif\left\lbrace 0,1,2\right\rbrace $\newline
Sample $V\sim U\left(0,1\right)$ independent of everything\newline
Given $Z$ and $K=k$ sample $\left(S_{1},\ldots,S_{n_{k}}\right)$ as follows:%
\newline
\If{ $Z=0$ }{
Sample $\tilde{w}=\left(S_{j}:\, j\leq n_{k}-1\right)$ from
$P_{k,0}\left(\cdot\right):=P\left(\cdot|\, A_{k}\right)$ \\
\eIf{ $V\leq\frac{3P\left(A_{k}\right)}{g\left(k\right)}I\left(A_{k},\, T_{m}\in\left[n_{k-1},n_{k}-1\right]\right)$}
{ $J=1$} {$J=0$}
}
\If{ $Z=1$ }{
Sample $\tilde{w}=\left(S_{j}:\, j\leq T_{m}\wedge\left(n_{k}-1\right)\right)$
from $P_{k,1}\left(\cdot\right)$
\[
dP_{k,1}\left(\tilde{w}\right)=\exp\left\{ \theta_{k}S_{T_{m}\wedge\left(n_{k-1}\right)}-\left(T_{m}\wedge\left(n_{k}-1\right)\psi_{k}\left(\theta_{k}\right)\right)\right\} dP\left(\tilde{w}\right)
\]
\eIf{ $V\leq\frac{3\exp\left\{ -\theta_{k}S_{T_{m}}+T_{m}\psi_{k}\left(\theta_{k}\right)\right\} }{g\left(k\right)}I\left(B_{k},\, A_{k}^{c},\, T_{m}\in\left[n_{k-1},n_{k}-1\right]\right)$}
{ $J=1$} {$J=0$}
}
\If{ $Z=2$ }{
Sample $\tilde{w}=\left(S_{j}:\, j\leq n_{k}-1\right)$ from
$P_{k,2}\left(\cdot\right):=P\left(\cdot|\, B_{k}^{c}\right)$ \\
\eIf{ $V\leq\frac{3P\left(B_{k}^{c}\right)}{g\left(k\right)}I\left(B_{k}^{c}\, A_{k}^{c},\, T_{m}\in\left[n_{k-1},n_{k}-1\right]\right)$}
{ $J=1$} {$J=0$}
}
\eIf{$J=1$}{Output $\left(J,\omega \right)$, where $\omega =\left(S_{j}\left(\mu\right):\, 1\leq j\leq T_{m}\right)$
\tcp {Recall: $S_j\left(\mu\right)=S_j -\mu j $. }} {Output $\left(J,\omega\right)$, where $\omega = \left[\, \right]$ and $J=0$. }
\end{algorithm}

\pagebreak

We now provide the following result which justifies the validity of the
algorithm.

\begin{proposition}
\label{Prop_Validity_Ber_Patch_up}The output $J$ is Bernoulli with success
parameter $P_{0}\left( T_{m}<\infty\right) $ and $\omega$ follows the
required distribution of $\left( S_{1},\ldots,S_{T_{m}}\right) $ given $%
T_{m}<\infty$. Moreover, if $E\left\vert X_{1}\right\vert ^{2+\varepsilon
}<\infty$, then the expected number of function evaluations required to
sample $J$ and $\omega$ is finite.
\end{proposition}

\begin{proof}
\medskip To verify that indeed $J\sim Ber(P_{0}(T_{m}<\infty))$, let $%
P^{\prime}\left( \cdot\right) $ denote the joint probability distribution of
$K$, $Z$, $(S_{1},...,S_{n_{K}-1})$, and $J$ induced by the algorithm. Note,
of course, that $n_{K}-1\geq T_{m}$ under $P^{\prime}\left( \cdot\right) $.
In addition, observe that%
\begin{equation}
\begin{array}{lll}
P^{\prime}\left( J=1|\,Z=0,K=k\right) & = & \frac{3P\left( A_{k}\right) }{%
g\left( k\right) }\cdot P_{0}\left( T_{m}\in\left[ n_{k-1},n_{k}-1\right]
|\,A_{k}\right) \\
&  &  \\
& = & \frac{3}{g\left( k\right) }\cdot P_{0}\left( T_{m}\in\left[
n_{k-1},n_{k}-1\right] ,\,A_{k}\right).%
\end{array}
\label{eq: J=1 cond Z=0 K=k}
\end{equation}
Let $r_{k,1}:=\exp(-\theta_{k}S_{T_{m}}+T_{m}\psi\left( \theta_{k}\right)
)I\left( B_{k},A_{k}^{c},T_{m}\in\left[ n_{k-1},n_{k}-1%
\right] \right) $, and define $E_{k,1}(\cdot)$ to be the expectation
operator associated to the exponential tilting distribution with parameter $%
\theta_{k}$ applied to the random variables $X_{1},...,X_{n_{k}-1}$ (see (\ref{eq def of P_k1})). Note that,%
\begin{equation}
\begin{array}{lll}
P^{\prime}\left( J=1|\,Z=1,K=k\right) & = & \frac{3}{g\left( k\right) }%
E_{k,1}\left[ r_{k,1}\right] \\
&  &  \\
& = & \frac{3}{g\left( k\right) }P_{0}\left( B_{k},A_{k}^{c},T_{m}\in\left[
n_{k-1},n_{k}-1\right] \right)%
\end{array}
\label{eq: J=1 cond Z=1 K=k}
\end{equation}
Finally,%
\begin{equation}
\begin{array}{lll}
P^{\prime}\left( J=1|\,Z=2,K=k\right) & = & \frac{3}{g\left( k\right) }%
P_{0}\left( B_{k}^{c},A_{k}^{c},T_{m}\in\left[ n_{k-1},n_{k}-1\right] \right)%
\end{array}
\label{eq: J=1 cond Z=2 K=k}
\end{equation}
Combining (\ref{eq: J=1 cond Z=0 K=k})-(\ref{eq: J=1 cond Z=2 K=k}) we have
\begin{equation}
\begin{array}{l}
P^{\prime}\left( J=1\right) = \\
=\sum\limits_{k=2}^{\infty}\frac{1}{3}\left( P^{\prime}\left(
J=1|\,Z=0,K=k\right) +P^{\prime}\left( J=1|\,Z=1,K=k\right) +P^{\prime
}\left( J=1|\,Z=2,K=k\right) \right) g\left( k\right) \\
=\sum\limits_{k=2}^{\infty}\left( P_{0}\left( T_{m}\in\left[ n_{k-1},n_{k}-1%
\right] ,\,A_{k}\right) +P_{0}\left( B_{k},A_{k}^{c},T_{m}\in\left[
n_{k-1},n_{k}-1\right] \right) +P_{0}\left( B_{k}^{c},A_{k}^{c},T_{m}\in%
\left[ n_{k-1},n_{k}-1\right] \right) \right) \\
=\sum\limits_{k=2}^{\infty}P_{0}\left( T_{m}\in\left[ n_{k-1},n_{k}-1\right]
\right) =P_{0}\left( T_{m}<\infty\right) . \\
\end{array}%
\end{equation}
\bigskip Similarly we can verify that if $J=1$, $\omega=\left(
S_{1},...,S_{T_{m}}\right) $ follows the conditional law $P\left( \omega\in
\cdot|T_{m}<\infty\right) $. Just note that for any $F$,

\begin{equation}
\begin{array}{ll}
P^{\prime}\left( \omega\in F,J=1|K=k,Z=0\right) & =P_{0}\left( \omega\in
F,T_{m}\in\left[ n_{k-1},n_{k}-1\right] |A_{k}\right) \cdot\frac{3P\left(
A_{k}\right) }{g\left( k\right) } \\
& =P_{0}\left( \omega\in F,T_{m}\in\left[ n_{k-1},n_{k}-1\right]
,A_{k}\right) \cdot\frac{3}{g\left( k\right) }, \\
P^{\prime}\left( \omega\in F,J=1|K=k,Z=1\right) & =P_{0}\left( \omega\in
F,T_{m}\in\left[ n_{k-1},n_{k}-1\right] A_{k}^{c}|B_{k}\right) \cdot \frac{%
3P\left( B_{k}\right) }{g\left( k\right) } \\
& =P_{0}\left( \omega\in F,T_{m}\in\left[ n_{k-1},n_{k}-1\right]
A_{k}^{c},B_{k}\right) \cdot\frac{3}{g\left( k\right) }, \\
P^{\prime}\left( \omega\in F,J=1|K=k,Z=2\right) & =P_{0}\left( \omega\in
F,A_{k}^{c},T_{m}\in\left[ n_{k-1},n_{k}-1\right] |B_{k}^{c}\right) \cdot%
\frac{3P\left( B_{k}^{c}\right) }{g\left( k\right) } \\
& =P_{0}\left( \omega\in F,T_{m}\in\left[ n_{k-1},n_{k}-1\right]
,B_{k}^{c},A_{k}^{c}\right) \cdot\frac{3}{g\left( k\right) }.%
\end{array}%
\end{equation}

Consequently, combining these terms%
\begin{equation}
\begin{array}{ll}
& P^{\prime}\left( \omega\in F,J=1\right) \\
& =\sum\limits_{k=2}^{\infty}[P_{0}\left( \omega\in F,T_{m}\in\left[
n_{k-1},n_{k}-1\right] ,A_{k}\right) +P_{0}\left( \omega\in F,T_{m}\in\left[
n_{k-1},n_{k}-1\right] A_{k}^{c},B_{k}\right) \\
& +P_{0}\left( \omega\in F,T_{m}\in\left[ n_{k-1},n_{k}-1\right]
,B_{k}^{c},A_{k}^{c}\right) ] \\
& =\sum\limits_{k=2}^{\infty}P_{0}\left( \omega\in F,T_{m}\in\left[
n_{k-1},n_{k}-1\right] \right) =P_{0}\left( \omega\in F,T_{m}<\infty\right) .%
\end{array}%
\end{equation}
Since $P^{\prime}\left( J=1\right) =P_{0}\left( T_{m}<\infty\right) $, we
conclude that indeed
\[
P^{\prime}\left( \omega\in F|J=1\right) =P_{0}\left( \omega\in
F|T_{m}<\infty\right) .
\]
We now argue that the expected number of function evaluations required to
generate $(J,\omega)$ has finite mean. Let us assume that sampling from $P_{k,0}\left( \cdot\right), P_{k,1}\left( \cdot\right) $, and $P_{k,2}\left( \cdot\right)$
takes $O(n_k)$ function evaluations (a fact that it is not difficult to see, but nonetheless we will justify in Subsections \ref{subsubpk0} and \ref{subsubpk1_B}). Then, we note that each proposal $\omega$
takes on the order of
\[
O(\sum_{k=2}^{\infty}n_{k}g\left( k\right) )\leq
O(\sum_{k=2}^{\infty}n_{k}^{2}P\left( Y>n_{k-1}\mu+m\right) )<\infty
\]
function evaluations; the sum is finite assuming that $\alpha>2$, as
indicated in (\ref{Cond_Alpha}).
\end{proof}

We close this section explaining how to sample from $P_{k,0}\left( \cdot
\right) ,P_{k,1}\left( \cdot \right) $, and $P_{k,2}\left( \cdot \right) $.
We will also verify that it takes $O(n_{k})$ function evaluations to sample $%
\omega $ in each of these three cases as claimed in the end of Proposition %
\ref{Prop_Validity_Ber_Patch_up}.

\subsubsection{Sampling from $P_{k,0}\left( \cdot \right) $ and $%
P_{k,2}\left( \cdot \right) $\label{subsubpk0}}

We now explain how to use Acceptance / Rejection to obtain a sample from $%
P_{k,0}\left( \cdot \right) $ (i.e. sampling $(S_{1},...,S_{n_{k}-1})$ given 
$A_{k}$). Our proposal distribution, which we denote by $Q\left( \cdot
\right) $, is based on a mixture $P\left( \cdot \right) $ and another
distribution which we denote by $\bar{P}\left( \cdot \right) $ to be
described momentarily. In particular, we shall set $Q=.5P+.5\bar{P}$. As we
shall see, the reason for introducing $P$ is to make sure that the
acceptance ratio is bounded uniformly over $\mu $. This will be relevant in
our discussion on mixing time in heavy-traffic in Section \ref{Sec:
Numerical} (i.e. when $\mu $ is close to zero). If $\mu $ is not close to
zero then we can simply select $Q=\bar{P}$ and the acceptance ratio will be
bounded uniformly in $k$, but not as $\mu \rightarrow 0$.

The distribution of $(S_{1},...,S_{n_{k}-1})$ under $\bar{P}\left( \cdot
\right) $ is better described algorithmically. First, we sample $T_{k}$ with
probability mass function $r_{k}\left( \cdot \right) $ given by%
\[
r_{k}\left( j\right) =\frac{P(X_{j}>\left( \mu j+m\right) ^{1-\delta })}{%
\sum_{j=n_{k-1}}^{n_{k}-1}P(X_{j}>\left( \mu j+m\right) ^{1-\delta })},
\]%
for $j\in \left\{ n_{k-1},\ldots ,n_{k}-1\right\} $. Next, given $T_{k}=j$,
sample $X_{j}$ conditional on $X_{j}>\left( \mu j+m\right) ^{1-\delta }$.
Finally, sample $X_{i}$, for $i\neq j$ and $1\leq i\leq n_{k}-1$ from the
nominal (unconditional) distribution. We then obtain that 
\[
\frac{d\bar{P}}{dP}\left( X_{1},...,X_{n_{k}-1}\right) =\frac{%
\sum_{j=n_{k-1}}^{n_{k}-1}I\left( X_{j}>\left( \mu j+m\right) ^{1-\delta
}\right) }{\sum_{j=n_{k-1}}^{n_{k}-1}P(X_{j}>\left( \mu j+m\right)
^{1-\delta })}.
\]%
Therefore, with $P_{k,0}\left( \cdot \right) =P\left( \cdot |A_{k}\right) $
we obtain that 
\begin{equation}
\begin{array}{ll}
\frac{I\left( A_{k}\right) }{P\left( A_{k}\right) }\cdot \frac{dP}{dQ}\left(
X_{1},...,X_{n_{k}-1}\right)  & =2\frac{I\left( A_{k}\right) }{P\left(
A_{k}\right) }\cdot \frac{\sum_{j=n_{k-1}}^{n_{k}-1}P(X_{j}>\left( \mu
j+m\right) ^{1-\delta })}{\sum_{j=n_{k-1}}^{n_{k}-1}I\left( X_{j}>\left( \mu
j+m\right) ^{1-\delta }\right) +\sum_{j=n_{k-1}}^{n_{k}-1}P(X_{j}>\left( \mu
j+m\right) ^{1-\delta })} \\ 
& \leq c_{k}:=\frac{2}{P\left( A_{k}\right) }\cdot \frac{%
\sum_{j=n_{k-1}}^{n_{k}-1}P(X_{j}>\left( \mu j+m\right) ^{1-\delta })}{%
1+\sum_{j=n_{k-1}}^{n_{k}-1}P(X_{j}>\left( \mu j+m\right) ^{1-\delta })}.%
\end{array}%
\end{equation}%
Consequently, in order to sample from $P_{k,0}\left( \cdot \right) $ it
suffices to propose from $Q\left( \cdot \right) $ and accept with probability%
\begin{eqnarray*}
q &:&=\frac{1}{c_{k}}\cdot \frac{I\left( A_{k}\right) }{P\left( A_{k}\right) 
}\cdot \frac{dP}{dQ}\left( X_{1},...,X_{n_{k}-1}\right)  \\
&=&I\left( A_{k}\right) \cdot \frac{1+\sum_{j=n_{k-1}}^{n_{k}-1}P(X_{j}>%
\left( \mu j+m\right) ^{1-\delta })}{\sum_{j=n_{k-1}}^{n_{k}-1}I\left(
X_{j}>\left( \mu j+m\right) ^{1-\delta }\right)
+\sum_{j=n_{k-1}}^{n_{k}-1}P(X_{j}>\left( \mu j+m\right) ^{1-\delta })}.
\end{eqnarray*}

We note that the expected number of proposals required to accept is $c_{k}$.
Moreover, as we shall quickly verify, $c_{k}$ is bounded uniformly both in $k
$ and $\mu >0$. To see this, use the fact that for $x\geq 0$, $1-x\leq \exp
\left( -x\right) $ and conclude that 
\begin{eqnarray*}
P\left( A_{k}\right) 
&=&1-\prod\limits_{j=n_{k-1}}^{n_{k}-1}(1-P(X_{j}>\left( \mu j+m\right)
^{1-\delta })) \\
&\geq &1-\exp \left( -\sum_{j=n_{k-1}}^{n_{k}-1}P(X_{j}>\left( \mu
j+m\right) ^{1-\delta })\right) .
\end{eqnarray*}%
Let us write 
\[
\Lambda :=\Lambda \left( k,\mu \right)
=\sum_{j=n_{k-1}}^{n_{k}-1}P(X_{j}>\left( \mu j+m\right) ^{1-\delta })
\]%
and therefore obtain that 
\[
c_{k}\leq \frac{2}{1-\exp \left( -\Lambda \right) }\cdot \frac{\Lambda }{%
1+\Lambda }\leq 4I\left( \Lambda \in \lbrack 0,1/2]\right) +6I\left( \Lambda
\geq 1/2\right) \leq 6.
\]

We suggest applying a completely analogous randomization procedure to sample 
$P_{k,2}\left( \cdot \right) $, which corresponds to sampling given the
event 
\[
B_{k}^{c}=\bigcup\limits_{j=1}^{n_{k}-1}\left\{ X_{j}>\left( \mu
n_{k-1}+m\right) ^{1-\delta }\right\} .
\]

A very similar argument as the one just discussed shows that the number of
proposals required to accept is also uniformly bounded over $k$ and $\mu $.
We therefore conclude that it takes $O(n_{k})$ function evaluations to
sample $\omega $ both under $P_{k,0}\left( \cdot \right) $ and $%
P_{k,2}\left( \cdot \right) $.

\subsubsection{Sampling from $P_{k,1}\left( \cdot \right) $ \label%
{subsubpk1_B}}

In order to simulate from $P_{k,1}\left( \cdot \right) $ we use Acceptance /
Rejection. We propose from $P\left( \cdot \right) $ (the nominal
distribution). Using the fact that $\theta _{k}=\gamma /C_{k}^{1-\delta }$,
note that%
\begin{equation}
\begin{array}{ll}
dP_{k,1} & =\frac{I\left( X\leq C_{k}^{1-\delta }\right) \exp \left( \theta
_{k}X-\psi _{k}\left( \theta _{k}\right) \right) }{P\left( X\leq
C_{k}^{1-\delta }\right) }dP \\
& \leq \frac{I\left( X\leq C_{k}^{1-\delta }\right) \exp \left( \gamma -\psi
_{k}\left( \theta _{k}\right) \right) }{P\left( X\leq C_{k}^{1-\delta
}\right) }dP\leq \frac{\exp \left( \gamma -\psi _{k}\left( \theta
_{k}\right) \right) }{P\left( X\leq C_{k}^{1-\delta }\right) }dP.%
\end{array}%
\end{equation}%
So, in order to sample from $P_{k,1}\left( \cdot \right) $ it suffices to
propose from $P\left( \cdot \right) $ and accept with probability%
\[
q\left( \omega \right) :=\frac{P\left( X\leq C_{k}^{1-\delta }\right) }{\exp
\left( \gamma -\psi _{k}\left( \theta _{k}\right) \right) }\frac{dP_{k,1}}{dP%
}=\exp \left( \theta _{k}X-\gamma \right) I\left( X\leq C_{k}^{1-\delta
}\right) .
\]%
The expected number of proposals required to obtain a successful sample $X$
from $P_{k,1}\left( \cdot \right) $ is equal to
\[
\frac{\exp \left( \gamma -\psi _{k}\left( \theta _{k}\right) \right) }{%
P\left( X\leq C_{k}^{1-\delta }\right) }\leq \frac{\exp \left( \gamma
\right) }{P\left( X\leq m\right) }<\infty ,
\]%
which is clearly uniformly bounded in $k$. So each increment takes $O\left(
1\right) $ time to be simulated and therefore we conclude it takes $O\left(
n_{k}\right) $ function evaluations to simulate $\omega $ under $%
P_{k,1}\left( \cdot \right) $.



\subsection{Building $M_{0}$ and $\left( S_{1}\left( \protect\mu\right)
,....,S_{\Delta}\left( \protect\mu\right) \right) $ from downward\ and
upward\ patches}

Before we move on to the algorithm let us define the following. Given a
vector $\mathbf{s}$, of dimension $d\geq1$, we let $\mathbf{L}(\mathbf{s})=%
\mathbf{s}\left( d\right) $ (i.e. the $d$-th component of the vector $%
\mathbf{s}$).

\begin{algorithm}
\caption{Sampling $M_{0}$ and $\left( S_{1}\left( \mu
\right),....,S_{\Delta} \left( \mu \right)\right)$} \label{algo Sampling
patches of path} 
\KwIn{Same as Algorithm \ref{algo Sampling M_0}}
\KwOut{
The path $\left( S_{1}\left( \mu \right),....,S_{\Delta }\left( \mu \right)
\right) $} 
Initialization\ $\mathbf{s}\leftarrow []$, $F\leftarrow 0$, $%
\mathbf{L}=0$ \newline
\tcp{initially $\mathbf{s}$ is the empty vector,the variable $\mathbf{L}$
represents the last position of the drifted random walk}
\While{$F=0$}  {
	Sample $\left(S_1\left(\mu \right), \ldots, S_{T_{-Lm}}\left(\mu
	\right)\right)$  given $S_0\left(\mu\right)=0$\newline
	$\mathbf{s}=\left[\mathbf{s},\mathbf{L}+ S_1\left(\mu \right), \ldots,\mathbf{L}+ S_{T_{-Lm}}\left(\mu \right)\right] $\newline
	$\mathbf{L}=\mathbf{L}\left(\mathbf{s}\right)$\newline
	Call Algorithm \ref{algo Sampling M_0} and obtain $\left(J,w\right)$ \newline
	\eIf{$J=1$}
	{ Set $\mathbf{s=[s, } \mathbf{L} + \omega ]$} {$F\leftarrow 1 \,\, \left( J=0\right) $}
  }

\textbf{
Output $\mathbf{s}$. }
\end{algorithm}

\begin{proposition}
\label{Prop_Mo_and_upto_Delta}The output of Algorithm \ref{algo Sampling
patches of path} has the correct distribution according to (\ref{eq def
Delta}) and (\ref{Eq_EVAL_M0}). Moreover, if $E\left\vert X_{1}\right\vert
^{2+\varepsilon}<\infty$, then the expected number of function evaluations
required to sample $M_{0}$ and $(S_{1}\left( \mu\right)
,....,S_{\Delta}\left( \mu\right) )$ is finite.
\end{proposition}

\begin{proof}
The fact that the output has the correct distribution follows directly from
our discussion leading to (\ref{Eq_EVAL_M0}) and from Proposition \ref%
{Prop_Validity_Ber_Patch_up}, which also implies that simulating a single
replication of $\left( J,\omega\right) $ using Algorithm \ref{algo Sampling
M_0} requires finite expected running time. But Algorithm \ref{algo Sampling
patches of path} requires a number of calls to Algorithm \ref{algo Sampling
M_0} which is geometrically distributed with mean $1/P_{0}\left(
T_{m}=\infty\right) <\infty$. Therefore, by Wald's identity (see \cite%
{Durrett}, p. 178) we conclude the finite expected running time of Algorithm %
\ref{algo Sampling patches of path}.
\end{proof}


\bigskip

 \section{From $M_{0}$ to $\left( S_{\lowercase{k}}\left( \protect\mu\right)
,M_{\lowercase{k}}\,:\lowercase{k}\geq0\right) $: Implementation of Procedure \ref{Algo
Hueristics for S_n M_n}}

\label{sec:Implementation-of-the algorithm}

In this section we will explain in detail how to implement the steps behind
the construction of the sequence $\left( S_{n} \left( \mu\right)
,M_{n}:\,n\geq0\right) $ that were described in Procedure \ref{Algo
Hueristics for S_n M_n}. We will be calling Algorithm \ref{algo Sampling M_0}
and Algorithm \ref{algo Sampling patches of path} repeatedly.


\subsection{Implementing Step \protect\ref{Hueristics Step_1} in Procedure
\protect\ref{Algo Hueristics for S_n M_n}}

\label{sub:Step-1}

In Step \ref{Hueristics Step_1} we need to sample a downward\ patch\ of the
drifted random walk $\left( S_{n}\left( \mu\right) :\,n\geq0\right) $. The
goal is to detect the time where the next downward milestone is crossed,
namely the next element in the sequence $(D_{n}:\,n\geq1)$, conditional on
the event that the level $C_{UB}$ is not crossed. To this end, let us invoke
a result in \cite{BS11}.

\begin{lemma}
\label{lem:Blanchet Sigman}Let $0<a<b\leq\infty$ and consider any sequence
of bounded positive measurable functions $f_{k}:\mathbb{R}%
^{k+1}\longrightarrow \mathbb{[}0,\infty)$.
\footnotesize {
\[
E_{0}\left[ f_{T_{-a}}\left( S_{0}\left( \mu\right) ,...,S_{T_{-a}}\left(
\mu\right) \right) |T_{b}=\infty\right] =\frac{E_{0}\left[ f_{T_{-a}}\left(
S_{0}\left( \mu\right) ,...,S_{T_{-a}}\left( \mu\right) \right)
I\left(S_i\left(\mu\right)<b,\, \forall i<T_{-a}\right)
P_{S_{T_{-a}}}\left( T_{b}=\infty\right) \right] }{P_{0}\left(
T_{b}=\infty\right) }
\]
}

So, if $P^{\ast }\left( \cdot \right) =P_{0}\left( \cdot |T_{b}=\infty
\right) $, then%
\begin{equation}
\frac{dP^{\ast }}{dP_{0}}=\frac{I\left( S_{i}\left( \mu \right) <b,\forall
i<T_{-a}\right) P_{S_{T_{-a}}}\left( T_{b}=\infty \right) }{P_{0}\left(
T_{b}=\infty \right) }\leq \frac{1}{P_{0}\left( T_{b}=\infty \right) }.
\end{equation}

\end{lemma}

\medskip{}

The result of Lemma \ref{lem:Blanchet Sigman} holds due to the strong Markov
property. Lemma \ref{lem:Blanchet Sigman} enables us to sample a downward
patch\ by means of the Acceptance/Rejection method using the nominal (i.e.
unconditional) distribution as proposal. More precisely, suppose that our
current position is $S_{D_{j}}\left( \mu\right) $ and we know that the
random walk will never reach position $C_{UB}$ (say, if $U_{j}=\infty$ then $C_{UB}=S_{D_{j}}\left( \mu\right) +m$).
Next we need to simulate the path
up to time $D_{j+1}$. Lemma \ref{lem:Blanchet Sigman} says that we can
propose a downward patch $s_{1}:=S_{1}\left( \mu\right)
,...,s_{T_{-Lm}}:=S_{T_{-Lm}}\left( \mu\right) $,  under the nominal
probability given $S_{0}\left( \mu\right) =0$ and 
$S_i\left(\mu\right)\leq m$ for $i\leq T_{-Lm}$. Then we accept the
downward patch with probability $P_{0}\left( T_{\sigma}=\infty\right) $,
where $\sigma=C_{UB}-S_{D_{j}}\left( \mu\right) -s_{T_{-Lm}}$.  
For example, 
if $U_{j}=\infty$ then $\sigma=m-s_{T_{-Lm}}\geq(L+1)m$.

Of course, to accept,  we can simulate a Bernoulli, say $B$, with probability $%
P_{0}\left( T_{\sigma }=\infty \right) $ by calling Algorithm \ref{algo
Sampling M_0} with $m\longleftarrow \sigma $ and returning $B=1-J$. If the
downward patch $\left( s_{1},...,s_{T_{-Lm}}\right) $ is accepted we
concatenate to produce the output%
\begin{eqnarray*}
&&\left( S_{0}\left( \mu \right) ,...,S_{D_{j}}\left( \mu \right)
,S_{D_{j}+1}\left( \mu \right) ,...,S_{D_{j+1}}\left( \mu \right) \right) \\
&=&(S_{0}\left( \mu \right) ,...,S_{D_{j}}\left( \mu \right)
,S_{D_{j}}\left( \mu \right) +s_{1},...,S_{D_{j}}\left( \mu \right)
+s_{T_{-Lm}}).
\end{eqnarray*}%
Otherwise, we keep simulating downward-patch proposals until acceptance.


\subsection{Implementing Step 2 in Procedure \protect\ref{Algo Hueristics
for S_n M_n}}

\label{sub:Step-2}

Assume we have finished generating the path up to time $D_{j+1}$ as
explained in Subsection \ref{sub:Step-1}. At this point we let $%
\sigma=C_{UB}-S_{D_{j+1}}\left( \mu\right) \geq(L+1)m$ and define

\begin{eqnarray*}
\xi &=&P_{0}\left( U_{j+1}<\infty |\,S_{1},\ldots ,S_{D_{j+1}},U_{1},\ldots
,U_{j}\right) \\
&=&P_{0}\left( T_{m}<\infty |\,T_{\sigma }=\infty \right) =P_{0}\left(
M_{0}>m\,|\,M_{0}\leq \sigma \right) .
\end{eqnarray*}%
\newline

Observe that assumption in equation (\ref{C_m_0}) ensures that $\xi>0$. We will explain
how to simulate $B\sim Ber\left( \xi\right) $. First, we call Algorithm \ref%
{algo Sampling patches of path} and obtain the output $\omega=\left(
s_{1},...,s_{\Delta}\right) $. We compute $M_{0}$ according to (\ref%
{Eq_EVAL_M0}) and keep calling Algorithm \ref{algo Sampling patches of path}
until we obtain $M_{0}\leq\sigma$, at which point we set $B=I\left(
M_{0}>m\right) $. Of course, we obtain $B\sim Ber\left( \xi\right) $ and if $%
B=1$ we can write%
\begin{equation}
\left( S_{D_{j+1}}\left( \mu\right) ,S_{D_{j+1}+1}\left( \mu\right)
,\ldots,S_{U_{j+1}}\left( \mu\right) \right) =(S_{D_{j+1}}\left( \mu\right)
,S_{D_{j+1}+1}\left( \mu\right) +s_{1},...,S_{D_{j}}\left( \mu\right)
+s_{\Delta}).  \label{Outp}
\end{equation}
Otherwise, $B=0$, we could simply declare $U_{j+1}=\infty$, update $%
C_{UB}\leftarrow S_{D_{j+1}}\left( \mu\right) +m$ and proceed to the next
iteration. 

Breaking the path into ``upward'' and ``downward'' patches helps to conceptualize the logic of our method.
However, it is not an efficient way of implementing the method.   
A more efficient implementation would be to sequentially generate versions of $\omega=(s_{1},...,s_{\Delta})$   as long
as $M_{0}\leq\ m$. 
We can then output the right hand side of (\ref{Outp}) even when $%
B=0$, because the path has been simulated according to the correct
distribution given $T_{\sigma}=\infty$. 
We provide a precise description of this implementation in Algorithm \ref{algo last} in the next section.

\subsection{Our algorithm to sample $\left( S_{k}\left( \protect\mu\right)
,M_{k}\,:0\leq k\leq n\right)$ and Proof of Theorem \protect\ref{Thm_MAIN}
\label{Sub_Algo_last}}

We close this section by giving the explicit implementation of our general
method outlined in Subsections \ref{sub:Step-1} and \ref{sub:Step-2}. In
order to describe the procedure, let us recall some definitions. Given a
vector $\mathbf{s}$ of dimension $d\geq1$, let $\mathbf{L}\left( \mathbf{s}%
\right) =\mathbf{s}\left( d\right) $ (the last element of the vector) and
set $\mathbf{d}\left( \mathbf{s}\right) =d$ (the length of the vector). The
implementation is given in Algorithm \ref{algo last}.

\begin{proof}[ of Theorem \protect\ref{Thm_MAIN}]
The validity of Algorithm \ref{algo last} is justified following the same
logic as in Proposition \ref{Prop_Mo_and_upto_Delta}. The only difference
here is that the number of trials required to simulate each upward patch is
geometrically distributed with a mean which is bounded by $1/P_{0}\left(
M_{0}=0 \right) <\infty $, following the reasoning behind (\ref%
{Good_Lower_Bound_p}). Also note that%
\[
E_{0}(T_{m}I(T_{m}<\infty ))\leq \sum_{k=2}^{\infty }n_{k}g\left( k\right)
<\infty .
\]%
Moreover, if $\sigma \geq (L+1)m$, by assumption (\ref{C_m_0})
\[
E_{0}\left( T_{m}|T_{m}<\infty ,T_{\sigma }=\infty \right) \leq \frac{%
E_{0}\left( T_{m}I(T_{m}<\infty \right) )}{P_{0}\left( T_{m}<\infty
,T_{\sigma }=\infty \right) }\leq \frac{E_{0}\left( T_{m}I(T_{m}<\infty
\right) )}{P_{0}\left( m<M_{0}\leq \sigma \right) }<\infty .
\]%
So, each upward path requires finite number of function evaluations to be
produced. The argument for finite expected running time then follows along
the lines of Proposition \ref{Prop_Mo_and_upto_Delta}.
\end{proof}

\begin{algorithm}[t]
\caption{An Efficient Implementation of Procedure 1}
\label{algo last}
\KwIn{Same as Algorithm \ref{algo Sampling M_0} and some $n\geq 1$}
\KwOut {$\left(S_{k}\left(\mu\right), M_k \, :0\leq k\leq n\right) $}

Initialization\  $\mathbf{s}\longleftarrow \lbrack 0]$,  $\mathbf{N}\longleftarrow \lbrack 0]$, $F\longleftarrow 0$
	\tcp {Initialize the sample path with the 1-dimensional zero vector.} 
	\tcp {The vector $N$, which is initially equals to zero records the times $D_{j}$ such that $U_{j}=\infty $}
	\tcp {$F$ is a Boolean variable which detects when we have enough information to compute $M_{n}$}

Call Algorithm \ref{algo Sampling patches of path} and obtain $\omega =(s_{1},...,s_{\Delta })$ 

Set $\mathbf{s=[s},\omega ]$
	\tcp { concatenate $\omega $ to $\mathbf{s}$} 

Set $\mathbf{N}=[\mathbf{N},\mathbf{d}\left( \mathbf{s}\right) -1]$  
	\tcp { update $\mathbf{N}$}

\While {$F=0$} 
{ 
	\eIf{$\mathbf{N}\left(\mathbf{d}\left(\mathbf{N}\right)-1\right)\geq n$}{
	$F=1$
	}
	{
	Call Algorithm \ref{algo Sampling patches of path} and  obtain $\omega =(s_{1},...,s_{\Delta })$, and compute $M_{0}$
		
		\If {$M_{0}\leq m$} 
		{
		Set $\mathbf{s=[s,} \mathbf{L}(\mathbf{s}) + \omega ]$
	
		Set $\mathbf{N}=[\mathbf{N},\mathbf{d}\left( \mathbf{s}\right) -1]$
		}
	
	} 

}

\For {$i=0,...,n$} 
{
$M_{i}=\max (\mathbf{s}\left( i+1\right) ,\mathbf{s}%
\left( i+2\right) ,....,\mathbf{s}\left( \mathbf{d(s)}\right) )-\mathbf{s}%
\left( i+1\right) $\\
$S_{i}\left( \mu \right) =\mathbf{s}\left( i+1\right) $
}

\medskip
\textbf{Output} $\left( S_{k}\left(\mu \right), M_k \,:1\leq k\leq n\right) $

\end{algorithm}

\bigskip{}

\section{Additional considerations: increments with infinite variance and
computing truncated tilted distributions}

\label{Sec_Add_Cons}

\subsection{Assuming that $E\left\vert X\right\vert ^{\protect\beta}<\infty$
for $\protect\beta\in(1,2]$}

We will now discuss how to relax the assumption that $E\left\vert
X\right\vert ^{\beta}<\infty$ for $\beta>2$ and assume only that $%
E\left\vert X\right\vert ^{1+\varepsilon}<\infty$ for $\varepsilon\in(0,1]$.

The development can be easily adapted. In order to facilitate the
explanation let us discuss the adaptation in the setting of Subsection \ref%
{Subsec_Discussion_Param}, which leads somewhat weaker bounds that those
assumed in (\ref{C_I_1}) to (\ref{C_I_2}) but strong enough to adapt the
conclusion in Lemmas \ref{lem 3PA_k over g_k leq 1} to \ref{lem 3PB_k^c over
g_k leq 1}.

In order to adapt equation (\ref{C_m_2}), for example, we now select $\delta
>0$ small enough so that $1<\alpha \leq \left( 1+\varepsilon \right) \left(
1-\delta \right) $. Then (\ref{C_m_2}) is replaced by
\[
\frac{6\cdot 2^{\alpha }}{\left( \alpha -1\right) \left( m+1\right) ^{\alpha
-1}\mu }\times E\left[ \left( X_{1}^{+}\right) ^{1+\varepsilon }\right] \leq
1.
\]%
These changes yield that inequality (\ref{C_I_1}), which in turn yields the
proof Lemma \ref{lem 3PA_k over g_k leq 1} and Lemma \ref{lem 3PB_k^c over
g_k leq 1}.

As for Lemma \ref{lem PB_k over g_k leq 1}, let us now apply Lemma \ref{lem
log MGF bound} with
\[
A\left( \gamma \right) =\left( \frac{\gamma ^{2}}{2}\cdot \frac{\exp \left(
1\right) }{1-\varepsilon }+2\right) \cdot E\left( \left\vert X\right\vert
^{1+\varepsilon }\right) ,
\]%
and obtain%
\begin{equation}
\exp (\psi _{k}\left( \theta _{k}\right) )\leq \exp \left( A\left( \gamma
\right) \frac{1}{C_{k}}\right) .  \label{eq bound log mgf exponantially}
\end{equation}%
Since $T_{m}$ we have that $S_{T_{m}}\geq \mu T_{m}+m$, and because $%
T_{m}\in \left[ n_{k-1},n_{k}-1\right] $ we conclude that%
\[
S_{T_{m}}\geq \mu n_{k-1}+m=C_{k}.
\]%
Therefore, on $T_{m}\in \left[ n_{k-1},n_{k}-1\right] $
\[
\exp (-\theta _{k}S_{T_{m}}+T_{m}\psi _{k}\left( \theta _{k}\right) )\leq
\exp (-\theta _{k}C_{k}+n_{k}\psi _{k}\left( \theta _{k}\right) )\leq \exp
(-\gamma C_{k}^{\delta }+A\left( \gamma \right) \frac{n_{k}}{C_{k}})\leq
\exp (-\gamma C_{k}^{\delta }+2A\left( \gamma \right) /\mu ),
\]%
where the last inequality was obtained from the bound $n_{k}/C_{k}\leq
n_{k}/(n_{k-1}\mu )$. So, we conclude, letting $z=\mu n_{k-1}$, that

\[
\frac{3\exp (-\gamma C_{k}^{\delta }+2A\left( \gamma \right) /\mu )}{g\left(
k\right) }\leq \frac{3\left( 2z+m\right) ^{\alpha }}{\left( \alpha -1\right)
\left( m+1\right) ^{\alpha -1}z}\exp \left( -\gamma \left( m+z\right)
^{\delta }+2A\left( \gamma \right) /\mu \right) .
\]%
Further, if $u=\gamma ^{1/\delta }(m+z)$, following the development in
Subsection \ref{Subsec_Discussion_Param}, we arrive at%
\begin{eqnarray*}
\frac{3\exp (-\gamma C_{k}^{\delta }+2A\left( \gamma \right) /\mu )}{g\left(
k\right) } &\leq &\frac{3\cdot 2^{\alpha }\gamma ^{-\alpha /\delta }}{\left(
\alpha -1\right) \left( m+1\right) ^{\alpha -1}\mu }\exp \left( 2A\left(
\gamma \right) /\mu \right) \max_{u\geq \gamma ^{1/\delta }m}u^{\alpha }\exp
\left( -u^{\delta }\right) \\
&\leq &\frac{3\cdot 2^{\alpha }\gamma ^{-\alpha /\delta }}{\left( \alpha
-1\right) \left( m+1\right) ^{\alpha -1}\mu }\exp \left( 2A\left( \gamma
\right) /\mu \right) \left( \frac{\alpha }{\delta }\right) ^{\alpha }\exp
\left( -\left( \frac{\alpha }{\delta }\right) ^{\delta }\right) .
\end{eqnarray*}%
For every $\gamma >0$ we can select $m$ large enough to make the right hand
side less than one and this yields the adaptation of the proof of Lemma \ref%
{lem PB_k over g_k leq 1} to the case $\beta \in (1,2]$.

This discussion implies that Algorithm \ref{algo last} provides unbiased
samples from $\left( M_{k},S_{k}\left( \mu\right) :0\leq k\leq n\right) $ in
finite time with probability one. Nevertheless, if $\varepsilon\in(0,1]$, we
have that $\alpha\leq\left( 1-\delta\right) \left( 1+\varepsilon\right) <2$
and therefore the expected number of function evaluations required to sample
$J$ in Algorithm \ref{algo Sampling M_0} is bounded from below by
\[
\sum_{k}n_{k}^{2}P\left( Y>\mu n_{k}+m\right) =\infty.
\]
Therefore, the expected running time of Algorithm \ref{algo last} is not
finite.

\subsection{The issue of evaluating $\protect\psi _{k}\left( \protect\theta %
_{k}\right) $\label{SubSam_P}}

We are concerned with the evaluation of (\ref{BND_PK2}), that is, during the
course of the algorithm we must decide if
\begin{equation}
V\leq 3\cdot \exp \left\{ -\theta _{k}S_{T_{m}}+T_{m}\psi _{k}\left( \theta
_{k}\right) \right\}I(T_{m}\in \left[ n_{k-1},n_{k}-1\right] ,A_{k}^{c},B_{k})  \label{IN_V_phi}
\end{equation}%
where $V\sim U\left( 0,1\right) $ independent of $S_{T_{m}}$ and $T_{m}$.
In order to decide if inequality (\ref{IN_V_phi}) holds one does not need to
compute $\eta _{k}:=\exp (\psi _{k}\left( \theta _{k}\right) )$ explicitly.
It suffices to construct a pair of monotone sequences $\{\eta _{k}^{+}\left(
n\right) :n\geq 0\}$ and $\{\eta _{k}^{-}\left( n\right) :n\geq 0\}$ such
that $\eta _{k}^{+}\left( n\right) \searrow \eta _{k}$ as $n\rightarrow
\infty $ and $\eta _{k}^{-}\left( n\right) \nearrow \eta _{k}$ as $%
n\rightarrow \infty $. It is important, however, to have the sequences
converging at a suitable speed. 
For example, it is not difficult to show that if%
\[
0\leq \eta _{k}^{+}\left( n\right) -\eta _{k}^{-}\left( n\right) \leq
c_{0}n^{-r}
\]%
for $r>2$, and the evaluation of $\eta _{k}^{+}\left( n\right) $, $\eta
_{k}^{-}\left( n\right) $ takes $O\left( l\left( k\right) n\right) $
function evaluations then the expected number of function evaluations
required to terminate Algorithm \ref{algo Sampling M_0} will be bounded if $%
\sum_{k}g\left( k\right) l\left( k\right) <\infty $ (this holds if $%
E\left\vert X\right\vert ^{\beta }<\infty $ for $\beta >2$ and $l\left(
k\right) =O\left( n_{k}\right) $, given our selection of $\alpha >2$). Note
the requirement on quadratic convergence ($r>2$). Sequences $\eta
_{k}^{+}\left( \cdot \right) $ and $\eta _{k}^{-}\left( \cdot \right) $ can
be constructed assuming the existence of a smooth density for $X$ using
quadrature methods. Nevertheless, we do not want to impose the existence of
a smooth density and thus we shall advocate a different approach for
estimating $\psi _{k}\left( \theta _{k}\right) $, based on coupling.

The approach that we advocate proceeds as follows. First, note that if $X$
has a lattice distribution, with span $h>0$, then $\psi_{k}\left( \theta
_{k}\right) $ can be evaluated with $O\left( C_{k}^{1-\delta}/h\right) $
function evaluations given $k$. So, the expected number of function
evaluations involved in implementing Algorithm \ref{algo last} and deciding (%
\ref{IN_V_phi}) is bounded, since $\sum g\left( k\right) C_{k}^{1-\delta
}=O(\sum g\left( k\right) n_{k})<\infty$.

Now, suppose that the distribution of $X$ is non-lattice. The idea is to
construct a coupling between $X_{j}\left( \mu\right) $ and a suitably
defined lattice-valued random variable $X_{j}^{\prime}\left( \mu^{\prime
}\right) $ so that $X_{j}\left( \mu\right) \leq X_{j}^{\prime}\left(
\mu^{\prime}\right) $, $EX_{j}^{\prime}=0$, and $\mu^{\prime}>0$. We will
simulate the random walk associated to the $X_{j}^{\prime}\left( \mu^{\prime
}\right) $'s, namely, $S_{j}^{\prime}\left( \mu^{\prime}\right) $ and the
associated sequence $\left( M_{j}^{\prime}:j\geq0\right) $, jointly with $%
\left( S_{j}\left( \mu\right) :0\leq j\leq n\right) $. Since $\max
\{S_{j}^{\prime}\left( \mu^{\prime}\right) :j\geq l\}\searrow-\infty$ as $%
l\rightarrow\infty$ we will be able to sample $(M_{k}:k\leq n)$ after
computing $N$ such that $\max\{S_{j}^{\prime}\left( \mu^{\prime}\right)
:j\geq N\}\leq\min\{S_{k}\left( \mu\right) :k\leq n\}$. We now proceed to
describe this strategy in detail.

Given $h>0$ define $X_{j}^{\prime}=h\lfloor X_{j}/h\rfloor-E(h\lfloor
X_{j}/h\rfloor)$; we omit the dependence on $h$ in $X_{j}^{\prime}$ for
notational convenience. In addition, let $\mu^{\prime}=\mu-E(h\lfloor
X_{j}/h\rfloor)-h$. Since $E(h\lfloor X_{j}/h\rfloor)<0$ for each $h>0$, if
we also select $h\leq\mu$ we have $\mu^{\prime}>0$. Define
\[
X_{j}^{\prime}\left( \mu^{\prime}\right) =X_{j}^{\prime}-\mu^{\prime
}=h\lfloor X_{j}/h\rfloor-\mu+h,
\]
and note that
\[
X_{j}^{\prime}\left( \mu^{\prime}\right) \geq X_{j}\left( \mu\right) .
\]
We then define the corresponding random walks $S_{n}^{\prime}=X_{1}^{\prime
}+...+X_{n}^{\prime}$, $S_{n}^{\prime}\left( \mu^{\prime}\right)
=S_{n}^{\prime}-n\mu^{\prime}$ with $S_{0}^{\prime}=0$ and
\[
M_{n}^{\prime}\left( \mu^{\prime}\right) =\sup\{S_{k}^{\prime}\left(
\mu^{\prime}\right) :k\geq n\}-S_{n}^{\prime}\left( \mu^{\prime}\right) .
\]
The following algorithm summarizes our strategy to simulate $(S_{k}\left(
\mu\right) ,M_{k}:0\leq k\leq n)$ when $\psi_{k}\left( \theta_{k}\right) $
cannot be computed exactly.

\begin{algorithm}
\caption{Strategy for simulating $(S_{k}\left(
\mu \right) ,M_{k}:0\leq k\leq n)$ } \label{Procedure sampling path
trancated increments}
\KwIn{Same as Algorithm \ref{algo Sampling M_0} but
for $X_{j}^{\prime}$ and $h\in(0,\mu)$ } \KwOut
{$\left(S_{k}\left(\mu\right), M_k \, :1\leq k\leq n\right) $} Call
Algorithm \ref{algo last} and obtain $\omega^{\prime}= (S_{k}^{\prime
}\left( \mu ^{\prime }\right) ,M_{k}^{\prime }:0\leq k\leq n)$\newline
Given $\omega^{\prime}= \left( S_{k}^{\prime }\left( \mu ^{\prime }\right)
:0\leq k\leq n\right) $ sample $\omega = (S_{k}:0\leq k\leq n)$ \tcp* {this
is done by sampling $X_{k}$ given the simulated outcome of $\lfloor
X_{k}/h\rfloor $} Set $M_{n}^{-}:=\min (S_{k}\left( \mu \right) :0\leq k\leq
n)$\newline
Using Algorithm \ref{algo last}, continue sampling $(S_{k}^{\prime }\left(
\mu ^{\prime }\right) ,M_{k}^{\prime }:n\leq k\leq N)$, where $N=\inf
\{k\geq n:M_{k}^{\prime }+S_{k}^{\prime }\left( \mu ^{\prime }\right) \leq
M_{n}^{-}\} $\newline
Given $\left( S_{k}^{\prime }\left( \mu ^{\prime}\right) :n\leq k\leq
N\right) $, sample $\left( S_{k}:n\leq k\leq N\right) $\newline
Set $M_{k}=\max \{S_{j}\left( \mu \right) :k\leq j\leq N\}-S_{k}\left( \mu
\right) $ for $0\leq k\leq n$\newline
\textbf{Output} $\left( S_{k}\left( \mu \right) ,M_{k}:0\leq k\leq n\right) $%
.
\end{algorithm}

The complexity analysis (i.e. finite expected running time if $E\left\vert
X_{1}\right\vert ^{2+\varepsilon}<\infty$) carries over since $%
EM_{0}^{\prime }<\infty$, $E|\min\{S_{k}(\mu):k\leq n\}|<\infty$, and
therefore $EN<\infty$, with $N$ defined in Algorithm \ref{Procedure sampling
path trancated increments}.

\bigskip 

\section{Numerical Example}

\label{Sec: Numerical}

We will now illustrate our algorithm by revisiting the example that was
described in the Introduction. 
This example considers the waiting time
sequence that corresponds to the single-server queue. 
Recall that this
sequence $\left(W_n:\,n\geq0\right)$ can be generated by the
so-called Lindley's recursion
\begin{equation}  \label{Lindley's recursion}
W_n = \left(W_{n-1}+X_{n}-\mu\right)^{+}
\end{equation}
and when in steady state, the $W_n$'s are equal in distribution to
\[
M_0=\max\limits_{m\geq0}\lbrace S_{m} \left(\mu\right)\rbrace
\]
To demonstrate the capability of our algorithm, we chose 
a sequence of $X_n$'s of the form
\begin{equation}  \label{eq example increamet def}
X_n=h\left\lfloor\frac{c}{h}V_n\right\rfloor-E\left(h\left\lfloor\frac{c}{h}%
V_n\right\rfloor\right)=:Y_n-E\left(Y_n\right)
\end{equation}
where $V_n \sim Pareto\left(\alpha'\right)$, that is,
$$
P\left(V>t\right)=\frac{1}{\left(1+t\right)^{\alpha'}} \hspace{1cm} t>0
$$
The parameters $\alpha'$, $c$, and $h$ can be changed in order to test the algorithm in different scenarios. 
$\alpha'>2$ determines how heavy the tail of the increments is, 
$h>0$ is the lattice parameter (the non-lattice case is where $h\rightarrow 0$), 
and $c>0$ controls the mean
of $Y_n$.

\subsection{Choice of Parameters}

As mentioned at the end of Subsection \ref{Subsec_Discussion_Param}, we used
the Excel solver in the following way: given our selection of $\alpha \in \left( 2,4\right) $, we
picked $\delta \in (0,1/2]$, $\gamma \geq 0$, and $m\geq 0$ so as to minimize
$m$ subject to (\ref{C_I_1}) and (\ref{C_I_2}). The input parameters $\mu $,
$\alpha' $, $h$, and $c$ 
are chosen to
test conditions ranging from light to heavy traffic (controlled primarily by
the parameter $\mu $), and from heavy tails to relatively lighter tails
(which are controlled by the parameter $\alpha' $).

We conclude our discussion by providing a formal comparison against the
relaxation time of the Markov chain $\{W_{n}:n\geq 0\}$ in heavy-traffic.
We chose a formal comparison because a rigorous computation of the exact relaxation time of the single-server queue is not available (to the best of
our knowledge) at the level of generality at which our algorithm works, although bounds have been studied, as is the case in \cite{FS06}.
We have argued that our algorithm is sharp, in the sense that it is applicable
under close to minimal conditions required for the stability of the
single-server queue. We believe that the heavy-traffic analysis provides yet
another interesting perspective.

Assuming that $\beta >2$ (i.e. the increments have finite variance), in
heavy traffic, as $\mu \rightarrow 0$, it is well known that at temporal
scales of order $O\left( 1/\mu ^{2}\right) $ and spatial scales of order $%
O\left( 1/\mu \right) $ Lindley's recursion can be approximated by a one
dimensional reflected Brownian motion (RBM). 
In fact, the approximation
persists also for the corresponding stationary distribution (which converges
after proper normalization to an exponential distribution, which is the
stationary distribution of RBM (see \cite{KA61}, for example)). The relaxation time of $\{W_{n}:n\geq
0\}$ is of order $O\left( 1/\mu ^{2}\right) $ as $\mu \rightarrow 0$.

The running time analysis of our algorithm involves the ``downward" patches, which take $O\left( m\right) $ random
numbers to be produced. We also need to account for the simulation of the
Bernoulli trials for each ``upward" patch,
which requires the generation of $K$ under $g\left( \cdot \right) $, and a
total of $C_{0}=O(\sum_{k=1}^{\infty }n_{k}g\left( k\right) )$ expected
random numbers to be simulated. This analysis holds because the number
of proposals required to sample $P_{k,0}$, $P_{k,1}$ and $P_{k,2}$ remains
bounded also as $\mu \rightarrow 0$. 
Therefore, the actual $X_{i}$'s conditional on the $E_{i}
$'s can be easily simulated. A similar strategy can be implemented for $%
P_{k,2}$.

Consequently, the over all cost of our algorithm is driven by $%
C_{0}=O\left( \mu ^{-2}m\right) $. We also need to ensure that $m$ is
selected so that (\ref{C_I_1}) and (\ref{C_I_2}) are satisfied. From the
analysis of (\ref{C_m_2}) and (\ref{C_m_4}), we see that $m=O\left( \mu
^{-1}\right) $ is always a possible choice. However, this choice can be
improved if one can select a large $\alpha $ , which in turn is feasible as
long as $z^{\alpha }P\left( X>z^{1-\delta }\right) =O\left( 1\right) $. In
particular, we can choose $m=O\left( 1/\mu ^{1/(\alpha -1)}\right) $,
provided that $\delta $ is chosen sufficiently close to unity in order to
satisfy (\ref{C_m_4}). Our exact sampling algorithm in heavy traffic has
a running time that is not worse that $O\left( 1/\mu ^{3}\right) $ and it
can be arbitrarily close to the relaxation time $O\left( 1/\mu ^{2}\right) $
of the chain $\{W_{n}:n\geq 0\}$.

\subsection{Simulation Results}

We tested the algorithm in four different cases in which we changed the nature of the random walk increments and the traffic intensity.
By picking $\alpha'=2.9$ and $\alpha'=7$, we considered heavy tailed increments and relatively lighter tailed increments, respectively. 
By changing the value of $c$, we changed the traffic intensity $\rho$, which is given by
$$
\rho =\frac{E\left(h\left\lfloor\frac{c}{h}V\right\rfloor\right)}
{E\left(h\left\lfloor\frac{c}{h}V\right\rfloor\right)+\mu}
\approx
\frac{cE\left(V\right)}
{cE\left(V\right)+\mu}
$$

Throughout all scenarios we used the parameters 
$$
L=1.1, \hspace{0.5cm}
h=0.1, \hspace{0.5cm}
\mu=1 \hspace{0.5cm}
\normalfont{and} \hspace{0.5cm}
\delta=0.38
$$

The rest of the parameters were chosen as follows:
\begin{center}
\begin{tabular}[h]{|l|l|l|l|l|l|l|l|l|}
\cline{1-9} 
  &
  \multicolumn{4}{|c|}{$\rho=0.3$}&
  \multicolumn{4}{|c|}{$\rho=0.8$}  
  \\
\cline{2-9}
&
$\alpha$ & $\gamma$ & $c$ & $m$  
&
$\alpha$ & $\gamma$ & $c$ & $m$ \\
\hline

\multirow{1} {*} {$\alpha'=7$}
& 
4 & 1.7 & 3 & 16 
&
4 & 0.75 & 25 & 217 \\
\hline

\multirow{1} {*} {$\alpha'=2.9$}
&
2.01 & 1.24 & 0.85 & 35 
&
2.01 & 0.74 & 8 & 400 \\ 
\hline
\end{tabular}
\end{center}

In each of the above cases we generated 100,000 exact replicas of $M_0$ and compared it with the chain $\left\lbrace W_n:\,0\leq n\leq l\right\rbrace$, where $l$ was picked to fit the scenario. 
To analyze the output of the chain, we used batches with varying sizes. In the light traffic case, for both $\alpha'=2.9$ and $\alpha'=7$, we used $l=10^6$ with batches of size $25$. 
In the heavy traffic scenario, we used $l=2\cdot 10^6$ with batches of size $50$ for $\alpha'=7$, and $l=4\cdot 10^6$ with batches of size $100$ for $\alpha'=2.9$

\noindent We summarized the result in the following table 
(see also Figure \ref{fig:Numeric Expectation}):

\begin{center}
\begin{tabular}[h]{|b{1.2cm} l|b{1.2cm}|b{1.2cm}|b{0.9cm}|b{1.2cm}|b{1.2cm}|b{0.9cm}| }
  \cline{1-8} 
    & &
  \multicolumn{3}{|c|}{\vspace{0.05cm} $\rho=0.3$}&
  \multicolumn{3}{|c|}{$\rho=0.8$}  
  \\
   
  \cline{3-8} 
  &
  & \vspace{0.2cm} LCI & UCI & RT 
  & LCI & UCI & RT \\ [0.2cm]
  \hline
 
 \multirow{2} {*} {$\alpha'=7$}
  & \multicolumn{1}{|l|}{Exact sampler} 
  & \vspace{0.2cm} 0.0709 & 0.0726
  & $\approx 1.5$
  & 10.9092 & 11.1159 
  & $\approx 10$ \\
 	[0.3cm] \cline{2-8} 
  & \multicolumn{1}{|l|}{Batch mean}
  & \vspace{0.2cm} 0.0701 & 0.0734 
  & $\approx 1$
  & 10.7542 & 11.1152 
  & $\approx 3$ \\
   [0.3cm]\hline

  \multirow{2}{*} {$\alpha'=2.9$} 
  &\multicolumn{1}{|l|}{Exact sampler} 
    & \vspace{0.2cm} 0.6505   & 0.7336
    & $\approx 3$  
    & 28.7925  & 29.6832
    & $\approx 15$     \\
   	[0.3cm] \cline{2-8} 
  &\multicolumn{1}{|l|}{Batch mean}
     & \vspace{0.2cm} 0.5344   & 0.7429
     & $\approx 1$
     & 28.7908  & 30.1681
     & $\approx 4$    \\
     [0.3cm]\hline
   	
\end{tabular}

\vspace{0.1cm}
LCI/UCI=Lower/Upper 95\% Confidence Interval. \hfill RT= Running Time (in minutes).
\end{center}

In the numerical examples we see that the IID replications of $M_0$ 
appear to be a reasonable approach to steady-state estimation, especially in light traffic. 
The performance deteriorates somewhat in heavy traffic, which is expected given our earlier discussion on running time in heavy traffic. 
Nevertheless, it is important to note that while our procedure does not have any bias, batch means do not provide control on the bias with absolute certainty. 
Overall, we feel that a few minutes of additional running time in exchange for total bias deletion is not an onerous price to pay.
Therefore, our procedure is not only of theoretical interest (as the first exact sampler for a general single-server queue), but of practical value as well.

\begin{figure}[tbp]
        \centering
 		\begin{minipage}{0.49\textwidth}
 		\begin{center}
 		$\alpha'=7$
 		\end{center}
 		\begin{minipage}[b]{0.45\textwidth}
 		\includegraphics[scale=0.14]{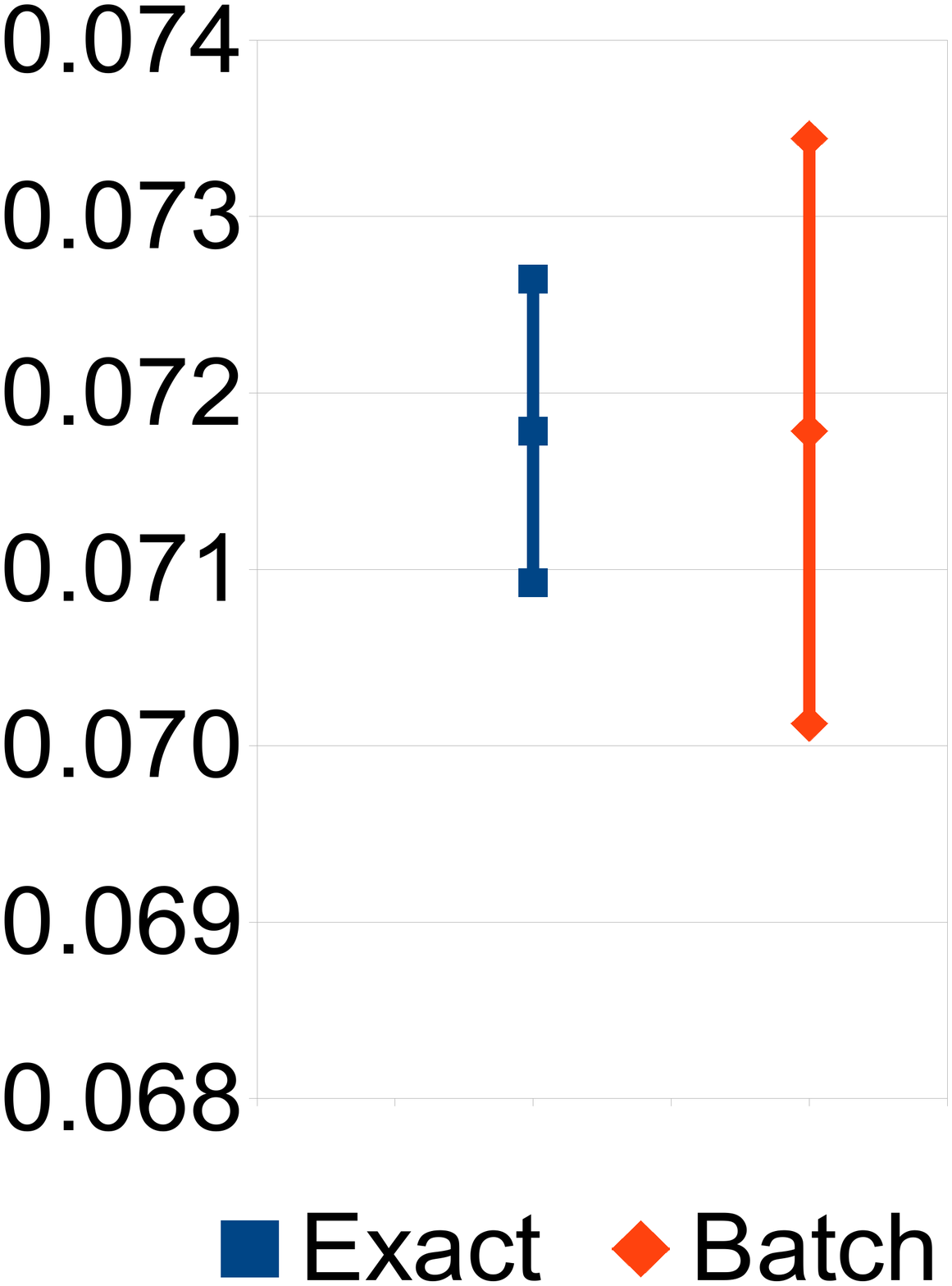}
 		\begin{center}
 		 $\rho=0.3$
 		\end{center}
 		\end{minipage}
 		\begin{minipage}[b]{0.45\textwidth}
 		\includegraphics[scale=0.14]{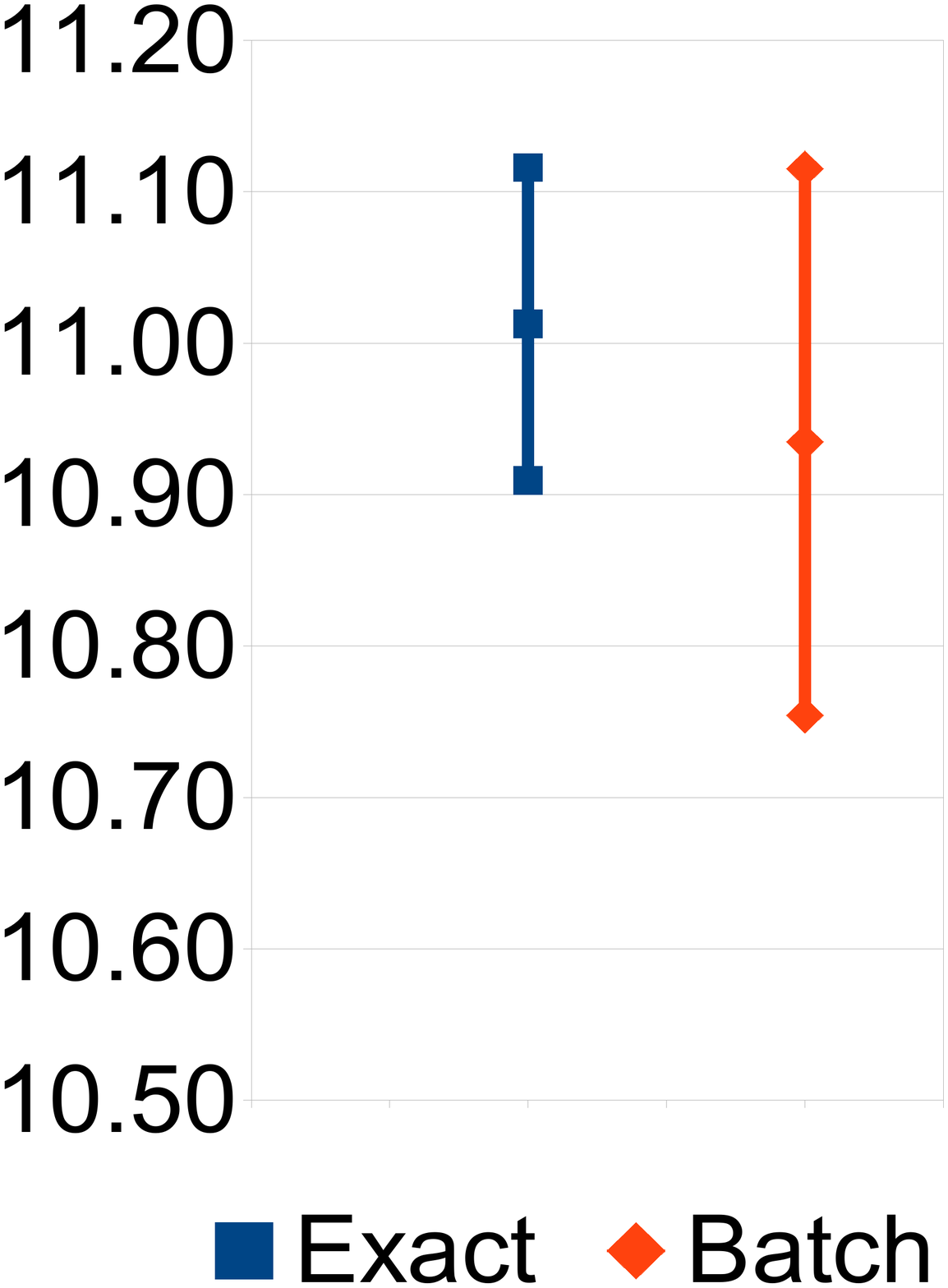}
 		\begin{center}
 		$\rho=0.8$
 		\end{center}
 		\end{minipage}
 		\end{minipage}      
        \begin{minipage}{0.49\textwidth}
         		\begin{center}
         		$\alpha'=2.6$
         		\end{center}
         		\begin{minipage}[b]{0.45\textwidth}
         		\includegraphics[scale=0.14]{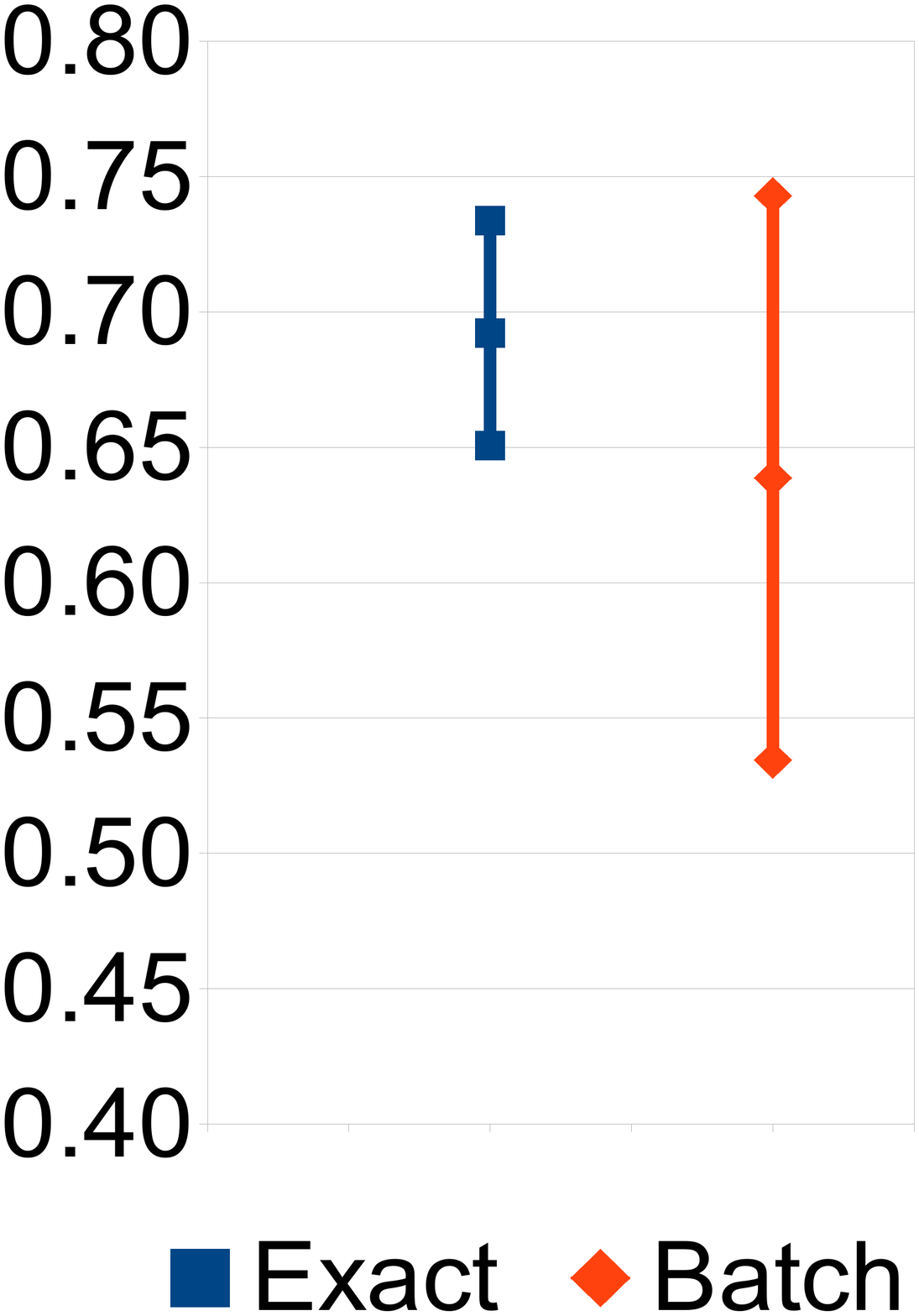}
         		\begin{center}
         		 $\rho=0.3$
         		\end{center}
         		\end{minipage}
         		\begin{minipage}[b]{0.45\textwidth}
         		\includegraphics[scale=0.14]{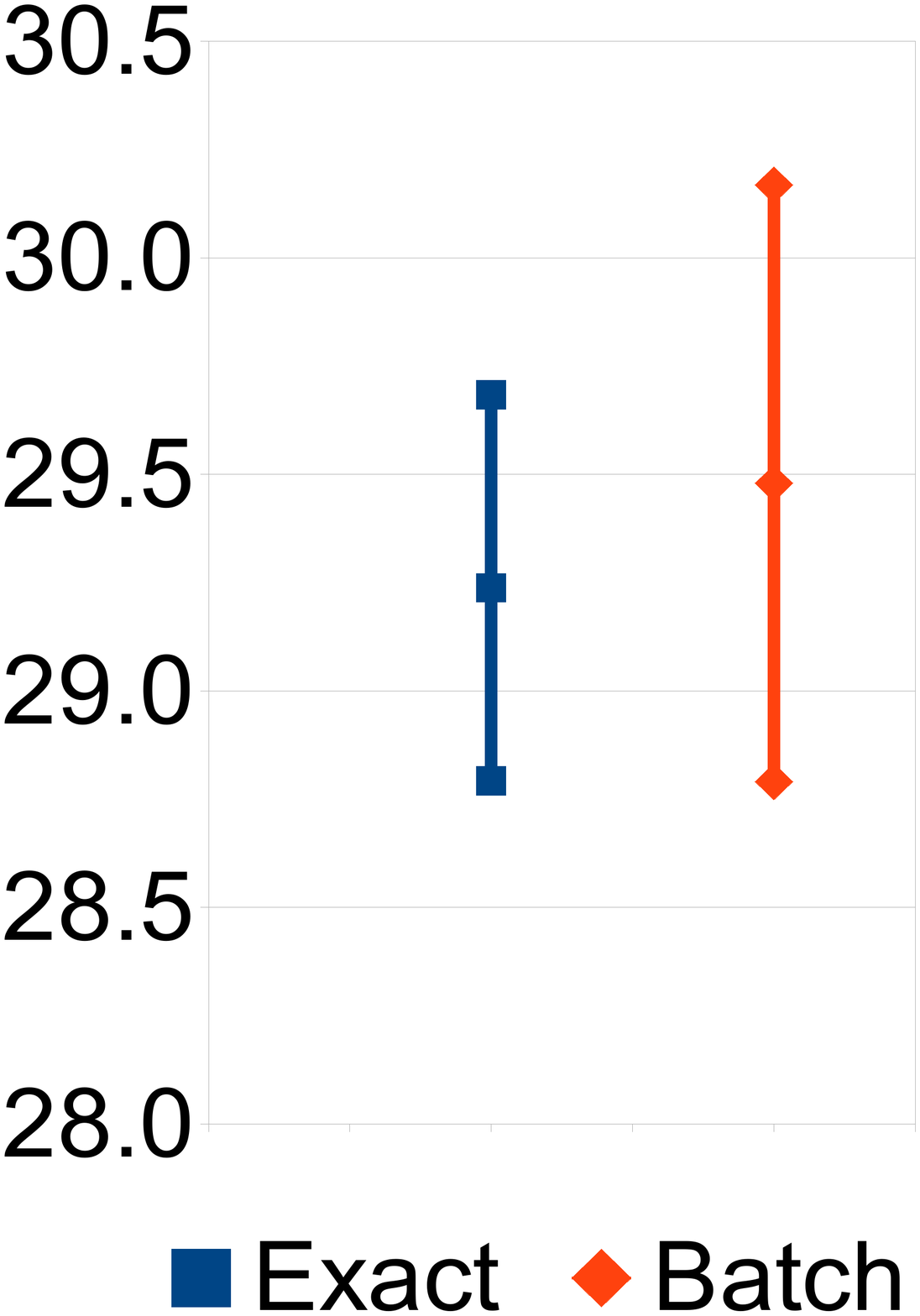}
         		\begin{center}
         		$\rho=0.8$
         		\end{center}
         		\end{minipage}
         		\end{minipage}
         		
\caption{Exact sampler mean $E\left(M_0\right)$ VS. batches mean of $\left\lbrace W_n:0\leq n\leq l\right\rbrace$, along with the corresponding 95\% confidence intervals.} 
\label{fig:Numeric Expectation}%
\end{figure}

\section{Conclusions}

The work presented in this paper was motivated by the important role that
single-server queue plays in many applications that use the DCFTP method
as well as the challenge of efficiently dealing with random walks involving
heavy-tailed increments. We developed an exact simulation method that can be
used to simulate the stationary waiting-time sequence of a single-server
queue backward in time, jointly with the input process of the queue. We
provided an algorithm, which is easy to implement, that has a finite
expected termination time under nearly minimal assumptions.

\medskip

\appendix

\section*{APPENDIX}

\section{Discussion on the generality of the assumptions imposed and
selection of parameters 
}

\label{Subsec_Discussion_Param}

In this section we will  argue that 
the inequalities (\ref{C_m_3})-(\ref{C_I_2})  can always be satisfied under
our underlying assumption that $E\left\vert X_{k}\right\vert ^{\beta
}<\infty $ for $\beta =2+\varepsilon >2$ (the case $\beta >1$ is discussed
in Section \ref{Sec_Add_Cons}). First, the selection of $L$ in (\ref{C_m_0})
is always feasible, as indicated earlier $L=1$ is most of the time feasible;
for example $L=1$ will be feasible if $X_{1}$ is non-lattice.

Clearly the selection of $m$ satisfying (\ref{C_m_3}) is always feasible.
Now, note that we can always select $\delta >0$ so that

\begin{equation}
2<\alpha \leq \left( 2+\varepsilon \right) \left( 1-\delta \right) .
\label{Cond_Alpha}
\end{equation}%
Then observe that if $m\geq 1$, applying Chebyshev's inequality,%
\begin{eqnarray*}
&&\frac{6\left( 1+2\mu z+m\right) ^{\alpha }}{\left( \alpha -1\right) \left(
m+1\right) ^{\alpha -1}\mu }P(X>\left( \mu z+m\right) ^{1-\delta }) \\
&\leq &\frac{6\cdot 2^{\alpha }\left( \mu z+m\right) ^{\alpha }}{\left(
\alpha -1\right) \left( m+1\right) ^{\alpha -1}\mu }\times \frac{E\left[
\left( X_{1}^{+}\right) ^{2+\varepsilon }\right] }{\left( \mu z+m\right)
^{\left( 2+\varepsilon \right) (1-\delta )}}\leq \frac{6\cdot 2^{\alpha }}{%
\left( \alpha -1\right) \left( m+1\right) ^{\alpha -1}\mu }\times E\left[
\left( X_{1}^{+}\right) ^{2+\varepsilon }\right] .
\end{eqnarray*}%
So, condition (\ref{C_I_1}) is automatically satisfied if $m$ is chosen
sufficiently large so that%
\begin{equation}
\frac{6\cdot 2^{\alpha }}{\left( \alpha -1\right) \left( m+1\right) ^{\alpha
-1}\mu }\times E\left[ \left( X_{1}^{+}\right) ^{2+\varepsilon }\right] \leq
1.  \label{C_m_2}
\end{equation}

Next, for (\ref{C_I_2}), we optimize over $z$ and obtain
\begin{equation}
\frac{z}{\left( m+z\right) ^{2(1-\delta )}}\leq \frac{1}{m^{1-2\delta }}%
\cdot \frac{\left( 1-2\delta \right) ^{1-2\delta }}{\left( 2\left( 1-\delta
\right) \right) ^{2\left( 1-\delta \right) }},  \label{C_Aux_B}
\end{equation}%
for all $\delta \in (0,1/2]$. Use Chebyshev's inequality, together with (\ref%
{C_Aux_B}), and the change of variable $u=\gamma ^{1/\delta }(m+z)$ to
obtain
\begin{eqnarray*}
&&\frac{3\left(1+2z+m\right) ^{\alpha }}{\left( \alpha -1\right) \left(
m+1\right) ^{\alpha -1}z}\exp \left( -\gamma \left( m+z\right) ^{\delta }+%
\frac{\gamma ^{2}e^{\gamma }E\left( X^{2}\right) z}{\left( m+z\right)
^{2(1-\delta )}\mu }+4\frac{z}{\mu }P\left( X>\left( z+m\right) ^{1-\delta
}\right) \right) \\
&\leq &\frac{3\left(1+2z+m\right) ^{\alpha }}{\left( \alpha -1\right) \left(
m+1\right) ^{\alpha -1}z}\exp \left( -\gamma \left( m+z\right) ^{\delta }+%
\frac{(\gamma ^{2}e^{\gamma }+4)E\left( X^{2}\right) \left( 1-2\delta
\right) ^{1-2\delta }}{\left( 2\left( 1-\delta \right) \right) ^{2\left(
1-\delta \right) }\mu m^{1-2\delta }}\right) . \\
&\leq &\frac{3\cdot 2^{\alpha }\gamma ^{-\alpha /\delta }}{\left( \alpha
-1\right) \left( m+1\right) ^{\alpha -1}\mu }\exp \left( \frac{(\gamma
^{2}e^{\gamma }+4)E\left( X^{2}\right) \left( 1-2\delta \right) ^{1-2\delta }%
}{\left( 2\left( 1-\delta \right) \right) ^{2\left( 1-\delta \right) }\mu
m^{1-2\delta }}\right) \max_{u\geq \gamma ^{1/\delta }m}u^{\alpha }\exp
\left( -u^{\delta }\right) .
\end{eqnarray*}%
Thus, we can first select $\gamma =1$, for example, and then pick the
smallest $m$ so that

\begin{equation}
\frac{3\cdot 2^{\alpha }}{\left( \alpha -1\right) \left( m+1\right) ^{\alpha
-1}\mu }\exp \left( \frac{7E\left( X^{2}\right) \left( 1-2\delta \right)
^{1-2\delta }}{\left( 2\left( 1-\delta \right) \right) ^{2\left( 1-\delta
\right) }\mu m^{1-2\delta }}\right) \max_{u\geq \gamma ^{1/\delta
}m}u^{\alpha }\exp \left( -u^{\delta }\right) \leq 1.  \label{C_m_4}
\end{equation}%
This can be done numerically or, explicitly by simply by noting (using
elementary calculus) that
\[
\max_{u\geq \gamma ^{1/\delta }m}u^{\alpha }\exp \left( -u^{\delta }\right)
\leq \left( \frac{\alpha }{\delta }\right) ^{\alpha }\exp \left( -\left(
\frac{\alpha }{\delta }\right) ^{\delta }\right) .
\]

In the numerical examples that we will discuss in Section \ref{Sec:
Numerical} we noted that the performance of the algorithm is not too
sensitive to the selection of $\alpha $, and thus we advocate picking $%
\alpha $ somewhat larger than 2, for instance $\alpha \in (2,4]$, but it is
important to constrain $\alpha $ and $\delta $ so that $z^{\alpha }P\left(
X>z^{1-\delta }\right) =O(1)$, due to (\ref{C_I_1}).

It is constraint (\ref{C_I_2}) the one that has the highest impact in the
algorithm's performance and we noted that the selection of $m$, in
particular, was the most relevant parameter. So, we simply used the Excel
solver; given our selection of $\alpha $ we picked $\delta \in (0,1/2]$, $%
\gamma \geq 0$ and $m\geq 0$ so as to minimize $m$ subject to (\ref{C_I_1})
and (\ref{C_I_2}). The optimization is done only once and it took a second.

In Section \ref{Sec: Numerical} we will also argue that the running time of
our algorithm is close to the relaxation time of the Markov chain from a
heavy-traffic perspective.

\section{Proof of Lemma \ref{lem 3PA_k over g_k leq 1}}

\label{appndx pf lem bound A_k}

\begin{proof}
Notice that
\[
\begin{array}{lll}
P\left( A_{k}\right) & \leq & \sum\limits_{j=n_{k-1}}^{n_{k}-1}P\left(
X_{j} >\left( j\mu +m\right) ^{1-\delta }\right) \\
& \leq & n_{k}P\left( X_{1}>\left( n_{k-1}\mu +m\right) ^{1-\delta }\right) .%
\end{array}%
\]%
It is straightforward to verify (using Chebyshev's inequality, the fact that
$E\left\vert X_{1}\right\vert ^{\beta }<\infty $ for $\beta >1$ and the
definition of $n_{k}$) that for any $\delta >0$,%
\[
\sum\limits_{k}n_{k}P\left( X_{1}>\left( n_{k-1}\mu +m\right) ^{1-\delta
}\right) <\infty .
\]%
Now we have for $k\geq 2$
\begin{equation}
\begin{array}{ll}
\frac{3P\left( A_{k}\right) }{g\left( k\right) } & \leq 3\bar{G}\left(
m\right) \frac{n_{k}P(X_{1}>\left( n_{k-1}\mu +m\right) ^{1-\delta })}{%
\int_{m+\mu n_{k-1}}^{m+\mu n_{k}}P\left( Y>s\right) ds} \\
& \leq 3\bar{G}\left( m\right) \frac{n_{k}P(X_{1}>\left( \mu
n_{k-1}+m\right) ^{1-\delta })}{\mu n_{k-1}P\left( Y>m+n_{k}\right) } \\
& =6\bar{G}\left( m\right) \frac{P(X_{1}^{+}>\left( \mu n_{k-1}+m\right)
^{1-\delta })}{\mu P\left( Y>m+\mu n_{k}\right) }\leq \frac{6\left(
1+2\mu n_{k-1}+m\right) ^{\alpha }}{\left( \alpha -1\right) \left( m+1\right)
^{\alpha -1}\mu }P\left( X>\left( \mu n_{k-1}+m\right) ^{1-\delta }\right)
\leq 1%
\end{array}%
\end{equation}%
Making $z=\mu n_{k-1}=\mu 2^{k-2}$ and using (\ref{C_I_1}) we obtain the
conclusion of the lemma.
\end{proof}


\section{Proof of Lemma \ref{lem PB_k over g_k leq 1}}

\label{Appndx pf lemma bound exp tilting}

Before we prove Lemma \ref{lem PB_k over g_k leq 1}, we will first introduce
an auxiliary lemma, which will be proved at the end of this section.

\begin{lemma} 
\label{lem log MGF bound} Set ${\theta =\gamma /u}^{1-\delta }$ \textup{\
for }$\delta \in \left( 0,1\right) $\textup{, }$u,\gamma >0$ and\textup{\
suppose that }$E\left( X\right) =0$. If $E\left( \left\vert X\right\vert
^{1+\varepsilon }\right) <\infty $ for some $\varepsilon \in \left(
0,1\right) $ and%
\begin{equation}
\frac{E\left( \left\vert X\right\vert ^{1+\varepsilon }\right) }{u^{\left(
1-\delta \right) \left( 1+\varepsilon \right) }}\leq \frac{1}{2},
\label{IN1}
\end{equation}%
then
\begin{equation}
E\left[ \exp (\theta X)\;|\;X\leq u^{1-\delta }\right] \leq \exp \left\{
\frac{A}{u^{\left( 1-\delta \right) \left( 1+\varepsilon \right) }}\right\}
\label{INB1}
\end{equation}%
with
\begin{equation}
A=\left( \frac{\gamma ^{2}}{2}\cdot \frac{\exp \left( \gamma \right) }{%
1-\varepsilon }+2\right) \cdot E\left( \left\vert X\right\vert
^{1+\varepsilon }\right) .  \label{IN1NV}
\end{equation}%
Moreover, if $E\left( X^{2}\right) <\infty $ and
\begin{equation}
\frac{E\left( X^{2}\right) }{u^{2\left( 1-\delta \right) }}\leq \frac{1}{2}
\label{IN2}
\end{equation}%
then%
\begin{equation}
E\left[ \exp (\theta X)\;|\;X\leq u^{1-\delta }\right] \leq \exp \left(
\frac{\gamma ^{2}\exp \left( \gamma \right) E(X^{2})}{2u^{2\left( 1-\delta
\right) }}+2P\left( X>u^{1-\delta }\right) \right) \leq \exp \left\{ \frac{A%
}{u^{2\left( 1-\delta \right) }}\right\} ,  \label{INB2}
\end{equation}%
with
\begin{equation}
A=\left( \frac{\gamma ^{2}\exp \left( \gamma \right) }{2}+2\right) \cdot
E\left( X^{2}\right) .  \label{IN2FV}
\end{equation}%
If in addition $u\geq 1$ and $0<\delta \leq \varepsilon /2$ then from (\ref%
{INB1}) we obtain
\begin{equation}
E\left[ \exp (\theta X)\;|\;X\leq u^{1-\delta }\right] \leq \exp \left(
\frac{A}{u}\right) ,  \label{INB3}
\end{equation}%
and if $EX^{2}<\infty $ inequality (\ref{INB3}) follows from (\ref{INB2})
choosing $0\leq \delta \leq 1/2$.
\end{lemma}

\bigskip

Having Lemma \ref{lem log MGF bound} at hand we are now ready to prove Lemma %
\ref{lem PB_k over g_k leq 1}

\bigskip

\begin{proof}[Proof of Lemma \protect\ref{lem PB_k over g_k leq 1}]
Since $m\geq 1$ satisfies inequality (\ref{C_m_3}), then we can invoke Lemma %
\ref{lem log MGF bound} with $u=n_{k-1}\mu +m=C_{k}$ and obtain%
\begin{equation}
\exp (\psi _{k}\left( \theta _{k}\right) )\leq \exp \left( \frac{\gamma
^{2}\exp \left( \gamma \right) E(X^{2})}{2C_{k}^{2\left( 1-\delta \right) }}%
+2P\left( X>C_{k}^{1-\delta }\right) \right) .
\label{eq bound log mgf exponantially}
\end{equation}%
By definition of $T_{m}$ we have that $S_{T_{m}}\geq \mu T_{m}+m$, and
because $T_{m}\in \left[ n_{k-1},n_{k}-1\right] $ we conclude that%
\[
S_{T_{m}}\geq \mu n_{k-1}+m=C_{k}.
\]%
Therefore, on $T_{m}\in \left[ n_{k-1},n_{k}-1\right] $
\begin{equation}
\exp (-\theta _{k}S_{T_{m}}+T_{m}\psi _{k}\left( \theta _{k}\right) )\leq
\exp (-\theta _{k}C_{k}+n_{k}\psi _{k}\left( \theta _{k}\right) ).
\label{eq_b_lik_ratio}
\end{equation}%
Combining (\ref{eq bound log mgf exponantially}) and (\ref{eq_b_lik_ratio}),
and letting $z=\mu n_{k-1}$, we obtain that%
\begin{eqnarray*}
&&\exp (-\theta _{k}S_{T_{m}}+T_{m}\psi _{k}\left( \theta _{k}\right) ) \\
&\leq &\exp \left( -\gamma \left( \mu n_{k-1}+m\right) ^{\delta }+\frac{%
\gamma ^{2}\exp \left( \gamma \right) E(X^{2})n_{k-1}}{\left( \mu
n_{k-1}+m\right) ^{2(1-\delta )}}+2n_{k}P\left( X>\left( \mu
n_{k-1}+m\right) ^{(1-\delta )}\right) \right) \\
&=&\exp \left( -\gamma \left( z+m\right) ^{\delta }+\frac{\gamma ^{2}\exp
\left( \gamma \right) E(X^{2})z}{\left( z+m\right) ^{2(1-\delta )}\mu }+4%
\frac{z}{\mu }P\left( X>\left( z+m\right) ^{(1-\delta )}\right) \right) .
\end{eqnarray*}%
So, using (\ref{C_I_2}) we conclude that
\begin{eqnarray*}
&&\frac{3\exp (-\theta _{k}S_{T_{m}}+T_{m}\psi _{k}\left( \theta _{k}\right)
)}{g\left( k\right) } \\
&\leq &\frac{3\left(1+2z+m\right) ^{\alpha }}{\left( \alpha -1\right) \left(
m+1\right) ^{\alpha -1}z}\exp \left( -\gamma \left( z+m\right) ^{\delta }+%
\frac{\gamma ^{2}\exp \left( \gamma \right) E(X^{2})z}{\left( z+m\right)
^{2(1-\delta )}\mu }+4\frac{z}{\mu }P\left( X>\left( z+m\right) ^{(1-\delta
)}\right) \right) \leq 1,
\end{eqnarray*}%
thereby obtaining the result.
\end{proof}

\bigskip

We conclude this appendix with the proof of the auxiliary lemma.

\bigskip

\begin{proof}[Proof of Lemma \ref{lem log MGF bound}]
Since $EX=0$, $E[XI\left( X\leq u^{1-\delta }\right) ]<0$, and therefore a
Taylor expansion of second order yields
\[
E\left[ \exp \left\{ X\frac{\gamma }{u^{1-\delta }}\right\} ,X\leq
u^{1-\delta }\right] \leq 1+\frac{\gamma ^{2}}{2}E\left[ \left( \frac{X}{%
u^{1-\delta }}\right) ^{2}\exp \left\{ \frac{\gamma X}{u^{1-\delta }}%
\right\} I\left( X\leq u^{1-\delta }\right) \right]
\]%
If $EX^{2}<\infty $, we conclude that
\[
E\left[ \exp \left\{ X\frac{\gamma }{u^{1-\delta }}\right\} ,X\leq
u^{1-\delta }\right] \leq 1+\frac{\gamma ^{2}\exp \left( \gamma \right) }{2}%
\cdot E(X^{2})\cdot \frac{1}{u^{2\left( 1-\delta \right) }}.
\]%
Since $1+x\leq \exp \left( x\right) $ for $x\geq 0$ we conclude that
\[
E\left[ \exp \left\{ X\frac{\gamma }{u^{1-\delta }}\right\} ,X\leq
u^{1-\delta }\right] \leq \exp \left( \frac{\gamma ^{2}\exp \left( \gamma
\right) }{2}\cdot E(X^{2})\cdot \frac{1}{u^{2\left( 1-\delta \right) }}%
\right) .
\]%
On the other hand%
\[
P\left( X\leq u^{1-\delta }\right) =1-P\left( X>u^{1-\delta }\right) \geq 1-%
\frac{E(X^{2})}{u^{2\left( 1-\delta \right) }}
\]%
and since $1-x\geq \exp \left( -2x\right) $ for $x\in \left( 0,1/2\right) $
we conclude that if (\ref{IN2}) holds then%
\[
E\left[ \exp \left\{ X\frac{\gamma }{u^{1-\delta }}\right\} \;|\;X\leq
u^{1-\delta }\right] \leq \exp \left( \frac{\gamma ^{2}\exp \left( \gamma
\right) E(X^{2})}{2u^{2\left( 1-\delta \right) }}+2P\left( X>u^{1-\delta
}\right) \right) ,
\]%
which yields (\ref{INB2}). Now, let's assume that $\varepsilon \in \left(
0,1\right) $ and $E\left\vert X\right\vert ^{1+\varepsilon }<\infty $. Since
$z^{2}\exp \left( -z\right) \leq 4\exp \left( -2\right) <1$ for $z\geq 0$ we
have that
\[
E\left[ \left( \frac{X\gamma }{u^{1-\delta }}\right) ^{2}\exp \left\{ \frac{%
X\gamma }{u^{1-\delta }}\right\} I\left( X\leq u^{1-\delta }\right) \right]
\leq \gamma ^{2}\exp \left( \gamma \right) E\left[ \left( \frac{X}{%
u^{1-\delta }}\right) ^{2}I\left( \left\vert X\right\vert \leq u^{1-\delta
}\right) \right] +P\left( X<-u^{1-\delta }\right) .
\]%
In addition,
\[
E\left[ \left\vert X\right\vert ^{2}I\left( \left\vert X\right\vert \leq
u^{1-\delta }\right) \right] =2E\left[ \int\limits_{0}^{u^{1-\delta
}}sI\left( \left\vert X\right\vert >s\right) ds\right] =2\int%
\limits_{0}^{u^{1-\delta }}sP\left( \left\vert X\right\vert >s\right) ds\leq
\frac{E\left\vert X\right\vert ^{1+\varepsilon }}{1-\varepsilon }%
u^{(1-\varepsilon )\left( 1-\delta \right) }\label{B1a}
\]%
Therefore,
\[
E\left[ \left( \frac{X}{u^{1-\delta }}\right) ^{2}I\left( \left\vert
X\right\vert \leq u^{1-\delta }\right) \right] \leq \frac{E\left\vert
X\right\vert ^{1+\varepsilon }}{1-\varepsilon }\cdot \frac{1}{u^{\left(
1+\varepsilon \right) \left( 1-\delta \right) }}.
\]%
Since%
\[
P\left( X<-u^{1-\delta }\right) \leq \frac{E\left\vert X\right\vert
^{1+\varepsilon }}{u^{\left( 1+\varepsilon \right) \left( 1-\delta \right) }}%
,\label{B1b}
\]%
we conclude combining (\ref{B1a})\ and (\ref{B1b}) that%
\begin{equation}
\begin{array}{ll}
E\left[ \exp \left\{ X\frac{\gamma }{u^{1-\delta }}\right\} ,X\leq
u^{1-\delta }\right] & \leq 1+\frac{\gamma ^{2}}{2}\cdot E\left\vert
X\right\vert ^{1+\varepsilon }\cdot \left( \frac{\exp \left( \gamma \right)
}{\left( 1-\varepsilon \right) }+1\right) \cdot \frac{1}{u^{\left(
1+\varepsilon \right) \left( 1-\delta \right) }} \\
& \leq 1+\gamma ^{2}\cdot E\left\vert X\right\vert ^{1+\varepsilon }\cdot
\frac{\exp \left( \gamma \right) }{\left( 1-\varepsilon \right) }\cdot \frac{%
1}{u^{\left( 1+\varepsilon \right) \left( 1-\delta \right) }}.%
\end{array}%
\end{equation}%
Similarly to the finite variance case we conclude that if (\ref{IN2}) holds,
then
\[
E\left[ \exp \left\{ X\frac{\gamma }{u^{1-\delta }}\right\} \;|\;X\leq
u^{1-\delta }\right] \leq \exp \left( \gamma ^{2}\cdot E\left\vert
X\right\vert ^{1+\varepsilon }\cdot \frac{\exp \left( \gamma \right) }{%
\left( 1-\varepsilon \right) }\cdot \frac{1}{u^{\left( 1+\varepsilon \right)
\left( 1-\delta \right) }}+2E\left\vert X\right\vert ^{1+\varepsilon }\cdot
\frac{1}{u^{\left( 1+\varepsilon \right) \left( 1-\delta \right) }}\right) ,
\]%
which in turn yields (\ref{INB1}). The last part of the result, namely (\ref%
{INB3}) follows from elementary algebra and the fact that we are requiring $%
u\geq 1$.
\end{proof}

\section{Proof of Lemma \ref{lem 3PB_k^c over g_k leq 1}}

\label{appndx pf lem bound B_k ^c}

\begin{proof}
Notice that
\[
\begin{array}{lll}
P\left( B_{k}^{c}\right) & \leq & \sum\limits_{j=n_{k-1}}^{n_{k}-1}P\left(
X_{j}>\left( j\mu+m\right) ^{1-\delta}\right) \\
& \leq & n_{k}P\left( X_{1}>\left( n_{k-1}\mu+m\right) ^{1-\delta}\right) .%
\end{array}
\]
Now we can continue and apply the same arguments as in Lemma \ref{lem 3PA_k
over g_k leq 1} to conclude the proof.
\end{proof}

\newpage
\bibliographystyle{plain}
\bibliography{Blanchet_Wallwater_Exact_Sampling_ref}

\begin{thebibliography}{10}

\bibitem{As03}
S.~Asmussen.
\newblock {\em Applied Probability and Queues, 2nd ed.}
\newblock Springer, New York, 2003.

\bibitem{BC11}
J.~Blanchet and X.~Chen.
\newblock Steady-state simulation for reflected {B}rownian motion and related
  networks.
\newblock {\em http://arxiv.org/pdf/1202.2062.pdf}, 2012.

\bibitem{BD12}
J.~Blanchet and J.~Dong.
\newblock Sampling point processes on stable unbounded regions and exact
  simulation of queues.
\newblock In {\em Proceedings of the 2012 Winter Simulation Conference C.
  Laroque, J. Himmelspach, R. Pasupathy, O. Rose, and A. M. Uhrmacher, eds},
  2012.

\bibitem{BS11}
J.~Blanchet and K.~Sigman.
\newblock On exact sampling of stochastic perpetuities.
\newblock {\em Journal of Applied Probability}, 48A:165--182, 2011.

\bibitem{CK14}
S.~Connor and S.~Kendall.
\newblock Perfect simulation of {M/G/c} queues.
\newblock {\em http://arxiv.org/abs/1402.7248}, 2014.

\bibitem{Durrett}
R.~Durrett.
\newblock {\em Probability: Theory and Examples}.
\newblock Duxbury Advanced Series, 2005.

\bibitem{EG00}
K.~B. Ensor and P.~W. Glynn.
\newblock Simulating the maximum of a random walk.
\newblock {\em Journal of Statistical Planning and Inference}, 85:127--135,
  2000.

\bibitem{FS06}
S.~Foss and A.~Sapozhnikov.
\newblock Convergence rates in monotone separable stochastic networks.
\newblock {\em Queueing Systems}, 52(2):125--137, 2006.

\bibitem{BDP15}
J.~Dong J.~Blanchet and Y.~Pei.
\newblock Perfect sampling of {GI/G/c} queues.
\newblock {\em Submitted}, 2015.

\bibitem{Ken98}
W.~S. Kendall.
\newblock Perfect simulation for the area-interaction point process.
\newblock In {\em Probability Towards 2000 (ed. L. Accardi and C.C. Heyde).
  Lecture Notes in Statistics}. volume 128, New York; Springer-Verlag, 218-234,
  1998.

\bibitem{KA61}
J.~F.~C. {Kingman} and M.~F. {Atiyah}.
\newblock {The single server queue in heavy traffic}.
\newblock {\em Proceedings of the Cambridge Philosophical Society}, 57:902,
  1961.

\bibitem{MJB13}
K.~Murthy, S.~Juneja, and J.~Blanchet.
\newblock State-independent importance sampling for random walks with regularly
  varying increments.
\newblock {\em http://arxiv.org/pdf/1206.3390.pdf}, 2013.

\bibitem{PW96}
J.~G. Propp and D.~B. Wilson.
\newblock Exact sampling with coupled {M}arkov chains and applications to
  statistical mechanics.
\newblock {\em Random Structures \& Algorithms (Atlanta, Georgia: Proceedings
  of the Seventh International Conference on Random Structures and
  Algorithms)}, 9:223--252, 1996.

\bibitem{Si11a}
K.~Sigman.
\newblock Exact simulation of the stationary distribution of the {FIFO M/G/c}
  queue.
\newblock {\em Journal of Applied Probability}, 48A:209--216, 2011.

\bibitem{Si11b}
K.~Sigman.
\newblock Exact simulation of the stationary distribution of the {FIFO M/G/c}
  queue: The general case for $\rho < c$.
\newblock {\em Queueing Systems}, 70:37--43, 2012.

\end{thebibliography}



@BOOK{As03,
  title = {Applied Probability and Queues, 2nd ed.},
  publisher = {Springer, New York},
  year = {2003},
  editor = {?},
  author = {S. Asmussen},
  owner = {Aya Wallwater},
  timestamp = {2012.07.25}
}

@BOOK{AG07,
  title = {Stochastic Simulation: Algorithms and Analysis},
  publisher = {Springer-Verlag},
  year = {2007},
  author = {S. Asmussen and P. W. Glynn},
  owner = {Aya Wallwater},
  timestamp = {2012.08.01}
}

@ARTICLE{BC11,
  author = {J. Blanchet and X. Chen},
  title = {Steady-state Simulation for Reflected {B}rownian Motion and Related
	Networks},
  journal = {http://arxiv.org/pdf/1202.2062.pdf},
  year = {2012},
  owner = {Aya Wallwater},
  timestamp = {2012.07.18}
}

@CONFERENCE{BD12,
  author = {J. Blanchet and J. Dong},
  title = {Sampling Point Processes on Stable Unbounded Regions and Exact Simulation
	of Queues},
  booktitle = {Proceedings of the 2012 Winter Simulation Conference C. Laroque,
	J. Himmelspach, R. Pasupathy, O. Rose, and A. M. Uhrmacher, eds},
  year = {2012},
  owner = {Aya Wallwater},
  timestamp = {2012.07.18}
}

@ARTICLE{BS11,
  author = {J. Blanchet and K. Sigman},
  title = {On Exact Sampling of Stochastic Perpetuities},
  journal = {Journal of Applied Probability},
  year = {2011},
  volume = {48A},
  pages = {165-182},
  owner = {Aya Wallwater},
  timestamp = {2012.07.18}
}

@ARTICLE{CK14,
  author = {Connor, S. and Kendall, S.},
  title = {Perfect Simulation of {M/G/c} Queues},
  journal = {http://arxiv.org/abs/1402.7248},
  year = {2014},
  owner = {Aya Wallwater},
  timestamp = {2012.07.18}
}

@ARTICLE{Duf90,
  author = {D. Dufresne},
  title = {The distribution of a perpetuity, with applications to risk theory
	and pension funding},
  journal = {Scandinavian Actuarial Journal},
  year = {1990},
  volume = {1990},
  pages = {39-79},
  owner = {Aya Wallwater},
  timestamp = {2012.07.30}
}

@BOOK{Durrett,
  title = {Probability: Theory and Examples},
  publisher = {Duxbury Advanced Series},
  year = {2005},
  author = {Durrett, R.},
  owner = {Aya Wallwater},
  timestamp = {2012.08.01}
}

@ARTICLE{EG00,
  author = {K. B. Ensor and P. W. Glynn},
  title = {Simulating the Maximum of a Random Walk},
  journal = {Journal of Statistical Planning and Inference},
  year = {2000},
  volume = {85},
  pages = {127-135},
  owner = {Aya Wallwater},
  timestamp = {2012.07.18}
}

@ARTICLE{FS06,
  author = {Foss, S. and Sapozhnikov, A.},
  title = {Convergence rates in monotone separable stochastic networks},
  journal = {Queueing Systems},
  year = {2006},
  volume = {52},
  pages = {125-137},
  number = {2},
  doi = {10.1007/s11134-006-4261-z},
  issn = {0257-0130},
  keywords = {Convergence rates; Moments; Coupling; Renovating events; Harris ergodic
	Markov chain; Monotone separable network; Generalized Jackson network;
	Multiserver queue},
  language = {English},
  publisher = {Kluwer Academic Publishers},
  url = {http://dx.doi.org/10.1007/s11134-006-4261-z}
}

@ARTICLE{Har77,
  author = {J. M. Harrison},
  title = {Ruin problems with compounding assets},
  journal = {Stochastic Processes and their Applications},
  year = {1977},
  volume = {5},
  pages = {67-79},
  owner = {Aya Wallwater},
  timestamp = {2012.07.30}
}

@ARTICLE{BDP15,
  author = {J. Blanchet, J. Dong and Y. Pei},
  title = {Perfect Sampling of {GI/G/c} Queues},
  journal = {Submitted},
  year = {2015},
  owner = {Aya Wallwater},
  timestamp = {2012.07.18}
}

@ARTICLE{LKPR11,
  author = {K. Latuszynski, I. Kosmidis, O. Papaspiliopoulos and G. O. Roberts},
  title = {Simulating Events of Unknown Probabilities via Reverse Time Martingales},
  journal = {Random Structures and Algorithms},
  year = {2011},
  volume = {38},
  pages = {441-452},
  owner = {Aya Wallwater},
  timestamp = {2012.07.23}
}

@ARTICLE{KO94,
  author = {M. S. Keane and G. L. O'Brien},
  title = {A {B}ernoulli factory},
  journal = {ACM Transactions on Modeling and Computer Simulation},
  year = {1994},
  volume = {4},
  pages = {213-219},
  owner = {Aya Wallwater},
  timestamp = {2012.08.01}
}

@ARTICLE{Kel96,
  author = {O. Kella},
  title = {Stability and non-product form of stochastic fluid networks with
	Levy inputs},
  journal = {The Annals of Applied Probability},
  year = {1996},
  volume = {6},
  pages = {186-199},
  owner = {Aya Wallwater},
  timestamp = {2012.07.18}
}

@ARTICLE{KW96,
  author = {O. Kella and W. Whitt},
  title = {Stability and structural properties of stochastic fluid networks},
  journal = {Journal of Applied Probability},
  year = {1996},
  volume = {33},
  pages = {1169-1180},
  owner = {Aya Wallwater},
  timestamp = {2012.07.18}
}

@INCOLLECTION{Ken98,
  author = {W. S. Kendall},
  title = {Perfect simulation for the area-interaction point process},
  booktitle = {Probability Towards 2000 (ed. L. Accardi and C.C. Heyde). Lecture
	Notes in Statistics},
  publisher = {volume 128, New York; Springer-Verlag, 218-234},
  year = {1998},
  owner = {Aya Wallwater},
  timestamp = {2012.07.18}
}

@ARTICLE{KA61,
  author = {{Kingman}, J.~F.~C. and {Atiyah}, M.~F.},
  title = {{The single server queue in heavy traffic}},
  journal = {Proceedings of the Cambridge Philosophical Society},
  year = {1961},
  volume = {57},
  pages = {902},
  adsnote = {Provided by the SAO/NASA Astrophysics Data System},
  adsurl = {http://adsabs.harvard.edu/abs/1961PCPS...57..902K},
  doi = {10.1017/S0305004100036094}
}

@ARTICLE{MJB13,
  author = {K. Murthy and S. Juneja and J. Blanchet},
  title = {State-independent Importance Sampling for Random Walks with Regularly
	Varying Increments},
  journal = {http://arxiv.org/pdf/1206.3390.pdf},
  year = {2013},
  owner = {Aya Wallwater},
  timestamp = {2012.07.18}
}

@ARTICLE{NP05,
  author = {S. Nacu and Y. Peres},
  title = {Fast simulation of new coins from old},
  journal = {Annals of Applied Probability},
  year = {2005},
  volume = {15},
  pages = {93-115},
  owner = {Aya Wallwater},
  timestamp = {2012.07.23}
}

@ARTICLE{Nyr01,
  author = {H. Nyrhinen},
  title = {Finite and infinite time ruin probabilities in a stochastic economic
	environment},
  journal = {Stochastic Processes and their Applications},
  year = {2001},
  volume = {92},
  pages = {265-285},
  owner = {Aya Wallwater},
  timestamp = {2012.07.30}
}

@INCOLLECTION{CPY01,
  author = {P. Carmona, F. Petit and M. Yor},
  title = {Explonential functionals of {L}evy processes},
  booktitle = {Levy Processes},
  publisher = {Birkhauser, Boston, MA},
  year = {2001},
  owner = {Aya Wallwater},
  timestamp = {2012.07.30}
}

@ARTICLE{Pau98,
  author = {J. Paulsen},
  title = {Sharp conditions for certain ruin in a risk process with stochastic
	return on investments},
  journal = {Stochastic Processes and their Applications},
  year = {1998},
  volume = {75},
  pages = {135-148},
  owner = {Aya Wallwater},
  timestamp = {2012.07.30}
}

@ARTICLE{PW96,
  author = {J. G. Propp and D. B. Wilson},
  title = {Exact sampling with coupled {M}arkov chains and applications to statistical
	mechanics},
  journal = {Random Structures \& Algorithms (Atlanta, Georgia: Proceedings of
	the Seventh International Conference on Random Structures and Algorithms)},
  year = {1996},
  volume = {9},
  pages = {223-252},
  owner = {Aya Wallwater},
  timestamp = {2012.07.18}
}

@ARTICLE{AGH92,
  author = {S. Asmussen, P. Glynn and H. Thorison},
  title = {Stationarity Detection in the Initial Transient Problem},
  journal = {ACM Transactions�on�Modeling�and Computer�Simulation},
  year = {1992},
  volume = {2},
  pages = {???},
  owner = {Aya Wallwater},
  timestamp = {2012.07.18}
}

@ARTICLE{AFK03,
  author = {S. Asmussen, S. Foss, and D. Korshunov},
  title = {Asymptotics for sums of random variables with local subexponential
	behaviour},
  journal = {Journal of Theoretical Probability},
  year = {2003},
  volume = {16},
  pages = {489-518},
  owner = {Aya Wallwater},
  timestamp = {2012.07.18}
}

@ARTICLE{FTC98,
  author = {S.G. Foss, R.L. Tweedie and J.N. Corcoran},
  title = {Simulating the invariant measures of {M}arkov chains using backward
	coupling at regeneration times},
  journal = {Probability in the Engineering and Informational Sciences},
  year = {1998},
  volume = {12},
  pages = {303-320},
  owner = {Aya Wallwater},
  timestamp = {2012.07.18}
}

@ARTICLE{Si11b,
  author = {K. Sigman},
  title = {Exact simulation of the stationary distribution of the {FIFO M/G/c}
	queue: The general case for $\rho < c$},
  journal = {Queueing Systems},
  year = {2012},
  volume = {70},
  pages = {37-43},
  owner = {Aya Wallwater},
  timestamp = {2012.07.23}
}

@ARTICLE{Si11a,
  author = {K. Sigman},
  title = {Exact simulation of the stationary distribution of the {FIFO M/G/c}
	queue},
  journal = {Journal of Applied Probability},
  year = {2011},
  volume = {48A},
  pages = {209-216},
  owner = {Aya Wallwater},
  timestamp = {2012.07.23}
}

@ARTICLE{TC00,
  author = {R. Tweedie and J. Corcoran},
  title = {Perfect Sampling of Ergodic {H}arris Chains},
  journal = {Annals of Applied Probability},
  year = {2000},
  volume = {11},
  pages = {438-451},
  owner = {Aya Wallwater},
  timestamp = {2012.07.18}
}

@comment{jabref-meta: selector_keywords:}

@comment{jabref-meta: selector_journal:}

@comment{jabref-meta: selector_publisher:}

@comment{jabref-meta: selector_author:}

@comment{jabref-meta: selector_review:}

\end{document}